\documentclass[reqno]{amsart}
\usepackage{amssymb,amsmath,amsthm}
\usepackage{mathrsfs}
\usepackage{amscd}
\usepackage{amsfonts}
\usepackage{amssymb}
\usepackage{latexsym}
\usepackage{color}
\usepackage{esint}
\usepackage{graphicx}
\usepackage{float}
\graphicspath{{Figures/}}

\usepackage{todonotes}
\usepackage[colorlinks,linkcolor=blue,anchorcolor=red,citecolor=blue]{hyperref}
\usepackage{cleveref}
\usepackage[margin=1in]{geometry} 
\usepackage{marginnote}

\usepackage[makeroom]{cancel}

\newtheorem{theorem}{Theorem}

\newtheorem{definition}{Definition}
\newtheorem{lemma}{Lemma}
\newtheorem{proposition}[theorem]{Proposition}
\newtheorem{remark}{Remark}

\let\e=\varepsilon
\let\d=\delta
\let\h=v

\let\p=\partial

\let\O=\Omega

\numberwithin{equation}{section}

\let\hide\iffalse
\let\unhide\fi

\newcommand{\R}{\mathbb{R}}
\renewcommand{\P}{\mathbf{P}}

\newcommand{\be}{\begin{equation}}
\newcommand{\bm}{\begin{multline}}
\newcommand{\ee}{\end{equation}}
\newcommand{\dd}{\mathrm{d}}

\newcommand{\xb}{x_{\mathbf{b}}}
\newcommand{\tb}{t_{\mathbf{b}}}

\newcommand{\Bes}{\begin{eqnarray*}}
\newcommand{\Ees}{\end{eqnarray*}}
\newcommand{\Be}{\begin{equation} }
\newcommand{\Ee}{\end{equation}}

\def\p{\partial}

\def\O{\Omega}
\def\R{\mathbb{R}}
\def\d{\mathrm{d}}
\def\B{\begin{equation}}
\def\E{\end{equation}}
\def\BN{\begin{eqnarray*}}
\def\EN{\end{eqnarray*}}

\usepackage{cite}

\makeatletter
\@namedef{subjclassname@2020}{%
  \textup{2020} Mathematics Subject Classification}
\makeatother

\begin{document}
\title[BGK equation with diffuse boundary]{BGK model for rarefied gas in a bounded domain}

\author{Hongxu Chen}
\address[HC]{Department of Mathematics, The Chinese University of Hong Kong, Shatin, Hong Kong}
\email{hchen463@cuhk.edu.hk}


\author{Christian Klingenberg}
\address[CK]{Department of Mathematics, University of Würzburg, Emil-Fischer-Str. 40, 97074 Würzburg, Germany}
\email{christian.klingenberg@uni-wuerzburg.de}

\author{Marlies Pirner}
\address[MP]{Department of Mathematics, University of Münster, Einsteinstr. 62, 48149 M\"unster, Germany}
\email{marlies.pirner@uni-muenster.de}

\date{\today}



\maketitle

\textbf{Abstract:} We study the Bathnagar-Gross-Krook (BGK) equation in a smooth bounded domain featuring a diffusive reflection boundary condition with general collision frequency. We prove that the BGK equation admits a unique global solution with an exponential convergence rate if the initial condition is a small perturbation around the global Maxwellian in the $L^\infty$ space. {\color{black}For the proof, we utilize the dissipative nature from the linearized BGK operator and establish an $L^2$ coercive estimate. Next, we derive the a priori estimate by obtaining an $L^\infty$ bound on the nonlinear operator; this requires a delicate analysis to manage its intrinsic nonlinear structure. Finally, we establish the $L^\infty$ stability estimate and introduce sequential arguments for the nonlinear BGK operator, thereby concluding both well-posedness and positivity.}

\section{Introduction}
The dynamics of a monatomic gas without chemical reactions is known to be described by the celebrated Boltzmann equation. But the complicated structure of the collision operator has long been a major obstacle in developing efficient numerical methods \cite{Cercignani}. Under certain assumptions, the complicated interaction terms of the Boltzmann equation can be simplified by a so called BGK approximation, consisting of a collision frequency multiplied by the deviation of the distributions from local Maxwellians. This approximation is constructed in a way such that it has the same main properties of the Boltzmann equation namely conservation of mass, momentum and energy. In addition\textcolor{black}{,} it has an H-theorem with an entropy inequality leading to an equilibrium which is a Maxwellian. \textcolor{black}{Our interest in this kind of models comes from the fact that it is
 used a lot by engineers, chemists and physicists and in numerical applications, see for
example \cite{MonteferranteMelchionnaMarconi2014,PieracciniPuppo2007,Tantos}.  BGK  models give rise to eﬃcient numerical computations, which are asymptotic preserving, that is they remain eﬃcient even
approaching the hydrodynamic regime \cite{BennouneLemouMieussens2008,BernardIolloPuppo2015,CrestettoCrouseillesLemou2012,DimarcoPareschi2014,DimarcoMieussensRispoli2014,FilbetJin2010,PicoOrtizDosSantosPhilippi2007,chen2019numerical}.}

It is used in many applications and there exist many extensions to deal with gas mixtures, ellipsoid statistical (ES-BGK) models, polyatomic molecules, chemical reactions or quantum gases; see for example \cite{gross_krook1956,hamel1965,Greene,Garzo1989,Sofonea2001,Pirner,haack,Bobylev,AndriesAokiPerthame2002,Todorova,Groppi,Pirner5,Pirner6,Bisi,Bisi2,Pirner9}.

In this paper, we consider the initial-boundary value problem of the BGK equation in a smooth bounded domain $\Omega$ in $\R^3$:
\Be\label{1.1}
\p_t F + v\cdot \nabla_x F =  \nu (M(F)-F),
\Ee 
where $F=F(t,x,v)\geq 0$ stands for the velocity distribution function of gas particles with velocity $v=(v_1,v_2,v_3)\in \R^3$ at time $t\geq 0$ and position $x=(x_1,x_2,x_3)\in \Omega\subset \R^3$. $M(F)$ is the local Maxwellian defined as
\Be\label{local_maxwellian}
M(F)(t,x,v) = \frac{\rho(t,x)}{(2\pi T(t,x))^{3/2}} \exp \Big( -\frac{|v-U(t,x)|^2}{2T(t,x)}\Big),
\Ee
where $\rho,U$ and $T$ correspond to the macroscopic quantities given by the moments of $F$:
\begin{align*}
    &   \rho(t,x) = \int_{\mathbb{R}^3} F(t,x,v) \dd v, \\
    & \rho(t,x)U(t,x) = \int_{\mathbb{R}^3} F(t,x,v) v \dd v,\\
    & 3\rho(t,x) T(t,x) = \int_{\mathbb{R}^3} F(t,x,v) |v-U(t,x)|^2 \dd v.
\end{align*}

The collision frequency $\nu$ takes the following form: for some constants $\eta,\omega$:
\begin{align*}
    & \nu(x):= \rho^\eta T^\omega.
\end{align*}

From the numerical point of view, the BGK model considerably simplifies the situation. But mathematical analysis is not necessarily easier, because the relaxation operator involves more non-linearity compared to the bilinear collision operator of the Boltzmann equation. 
In \cite{Perthame1989}, Perthame established the global existence of weak solutions in whole space for the BGK model with constant collision frequency. Regularity and uniqueness were considered in \cite{Perthame} under the local existence framework in the torus. In a near-a-global-Maxwellian regime, the global existence in the whole space in $\mathbb{R}^3$ and a polynomial convergence to equilibrium was established in \cite{Bellouquid2003}. In \cite{Yun}, for a wide class of non-trivial collision frequencies, the existence of a unique global smooth solution is established in the torus under a close-to-equilibrium assumption on the initial data and an exponential decay estimate is established in a high order energy norm. There are also various extensions of the previous result to more complicated BGK-type equations as the BGK equation for gas mixtures in \cite{BGKmixtures}, the ellipsoidal BGK model \cite{ESBGK}, relativistic and quantum BGK models \cite{BGKrel,BGKquantum}. Moreover, a method to construct sharp convergence rates for the  BGK equation is given in \cite{10.1007/978-3-319-32144-8_1,sharp}. All results here, concerning exponential convergence to equilibrium are in the torus and use a high order energy method to show exponential convergence for the non-linear BGK equation in a close-to-equilibrium regime. {\color{black}On the other hand, there are very few studies on the boundary value problem of the BGK problem.

Reflective boundary conditions play a role in many applications. Therefore, several numerical methods for the BGK equation with reflective boundary conditions have been proposed in the literature, e.g. \cite{Groppi2016Boundary,RussoFilbet2009MovingBoundaryBGK,BernardIolloPuppo2015}, also focusing on approaches preserving at the discrete level the asymptotic limit towards Euler equations up to the wall, thus ensuring a smooth transition towards the hydrodynamic regime \cite{BernardIolloPuppo2015}.
Therefore, in this article, we aim to provide a theoretical foundation for the boundary value problem and construct a unique global solution to the BGK equation with the diffusive reflection boundary condition. We note that the diffusive boundary condition is one of the most important reflection-type boundary conditions, and it corresponds to the no-slip boundary condition in the hydrodynamic limit, cf. \cite{saint2009hydrodynamic}. 

\color{black}

}

In the presence of the boundary, due to the characteristic nature, the kinetic equation exhibits singularities near the boundary \cite{GKTT,GKTT2,K,CK,CK_2023,chen2019regularity}, the high order energy method and Fourier transform method(cf. \cite{duan2021global}) become unavailable. To address the challenges posed by the nonlinear BGK operator, in the paper we focus on constructing a low-regularity solution, specifically achieving $L^\infty_{x,v}$ control without relying on the embedding $H^2_x\subset L^\infty_x$. The linear BGK operator possesses a dissipative property for the microscopic components in the $L^2_{v}$ energy estimate, which allows us to manage the additional nonlinearity introduced by the BGK operator by seeking a solution in the space $L^2_{x,v}\cap L^\infty_{x,v}$. Guo proposed this $L^2_{x,v}\cap L^\infty_{x,v}$ framework in \cite{G}, which established global well-posedness and exponential convergence to the global Maxwellian for the Boltzmann equation including diffuse and specular boundary condition. This breakthrough has significantly advanced the study of the boundary value problem of the Boltzmann equation, we refer to \cite{KL} for the specular boundary and \cite{briant,CKQ,bernou2022hypocoercivity,chen2024boltzmann} for intermediate status between pure diffuse reflection and pure specular reflection. {\color{black}Our main purpose in this paper is to propose an effective method to construct the BGK solution in the low regularity space $L^2_{x,v} \cap L^\infty_{x,v}$. Thus we only focus on the classical diffuse reflection boundary condition as mentioned earlier. We expect that our methodology can be applied to investigate the relevant problems, such as the well-posedness theory under other boundary conditions, the stationary problem, the regularity issues, the hydrodynamic limits, etc.}

To the end, we denote the boundary of the phase space as
\begin{equation*}
\gamma:= \{(x,v)\in \p\O \times \mathbb{R}^3\}.
\end{equation*}
Let $n=n(x)$ be the outward normal direction at $x\in\p\O$. We decompose $\gamma$ as
\begin{equation*}
\begin{split}
    &\gamma_- = \{(x,v)\in \p\O\times \mathbb{R}^3 : n(x)\cdot v < 0\},  \\
    & \gamma_+ = \{(x,v)\in \p\O\times \mathbb{R}^3 : n(x)\cdot v > 0\}, \\
    &\gamma_0 = \{(x,v)\in \p\O\times \mathbb{R}^3 : n(x)\cdot v = 0\}.
\end{split}    
\end{equation*}
The diffusive reflection boundary condition is prescribed for the incoming phase space: 
\begin{align*}
  F(t,x,v)|_{\gamma_-}  &  = c_\mu\mu(v) \int_{n(x)\cdot u>0} F(t,x,u) (n(x)\cdot u) \dd u,  
\end{align*}
where $\mu$ corresponds to the normalized global Maxwellian:
\begin{align*}
\mu(v):= \frac{1}{(2\pi)^{3/2}} e^{-\frac{|v|^2}{2}}.
\end{align*}
The constant $c_\mu=\sqrt{2\pi}$ is chosen to satisfy $\int_{n(x)\cdot v<0 }c_\mu\mu(v) |n(x)\cdot v|dv=1$ so that $c_\mu\mu(v) |n(x)\cdot v|$ is a probability measure on the half velocity space $\{\R^3: n(x)\cdot v< 0\}$. Note that the mass flux is vanishing at the boundaries, namely
\begin{equation*}
\int_{\R^3} (n(x)\cdot v)F(t,x, v)\d v=0 , \ x\in \p\O.
\end{equation*}

We seek the solution around the global Maxwellian, which takes the form $F = \mu + \sqrt{\mu}f$. Then, the following equation for $f$ can be derived 
\begin{equation}
\label{f_eqn}
\left\{\begin{aligned}
  &\p_t f + v\cdot \nabla_x f +  \mathcal{L} f  =  \Gamma(f) \quad \text{in }(0,\infty)\times \Omega\times \R^3,\\
   &            f(t,x,v)|_{\gamma_-}       = c_\mu\sqrt{\mu(v)} \int_{n(x)\cdot u>0} f(t,x,u)\sqrt{\mu(u)} |n(x)\cdot u| \dd u \quad \text{for }x\in \p\O, \\
   & f(0,x,v) = f_0(x,v).
\end{aligned}\right.
\end{equation}
Here,  $\mathcal{L}$ is a linearized collision operator, and $\Gamma$ a nonlinear collision operator. 
To define these operators, we first denote the orthonormal basis 
\begin{align*}
    & \chi_0(v): = \sqrt{\mu(v)}, \ \ \chi_i(v): = v_i \sqrt{\mu(v)}, \ i=1,2,3, \ \ \chi_4(v): = \frac{|v|^2-3}{2}\sqrt{\mu(v)}.
\end{align*}
We denote $\mathbf{P}f$ as the macroscopic quantities, which is defined as the $L^2_v$ projection of $f$ onto the subspace spanned by $\chi_i$:
\begin{align*}
   \mathbf{P}f & := \sum_{i=0}^4 \langle f,\chi_i\rangle \chi_i =  a(x)\chi_0 + \sum_{i=1}^3 b_i(x)\chi_i + c(x) \chi_4,
\end{align*}
with
\begin{equation*}
   a(t,x)   : = \langle f,\chi_0\rangle ,  \ \ \mathbf{b}(t,x)=(b_1(t,x),b_2(t,x),b_3(t,x)), \ 
   b_i(t,x)   := \langle f,\chi_i\rangle
   \ \text{ for }
   i=1,2,3 ; \ \
   c(t,x)   := \langle f,\chi_4 \rangle, 
\end{equation*}
where we have taken the usual inner product on $L^2(\R^3_v)$: 
\begin{equation*}
 \langle f,g\rangle     = \int_{\mathbb{R}^3} f(v)g(v)\,\dd v.
\end{equation*}

The linear operator $\mathcal{L}$ is then defined as $\mathcal{L}f = (\mathbf{I}-\mathbf
P)f$. The nonlinear operator $\Gamma(f)$ is defined as the remaining term in the BGK operator \eqref{1.1}:
\begin{align}
    &    \Gamma(f):= \frac{\nu (M(\mu+\sqrt{\mu}f)-\mu-\sqrt{\mu}f)}{\sqrt{\mu}}-(\mathbf{I}-\mathbf{P})f. \label{Gamma_def}
\end{align}
Here we highlight that $\Gamma(f)$ is a nonlinear operator of $f$, which exhibits a higher degree of nonlinearity compared to the bilinear Boltzmann operator. The derivation of $\mathcal{L}f$ and the explicit expression of $\Gamma(f)$ can be obtained by performing a Taylor expansion around the equilibrium state $(\rho,u,T) = (1,0,1)$.
For the detailed derivation and the associated properties, we refer to the next section (\eqref{linear_op} and \eqref{nonlinear_op} in Lemma \ref{lemma:operator_property}).

{\color{black}
We denote a velocity weight as 
\begin{equation}\label{weight}
w(v):= (1+|v|)^{\beta}e^{\theta |v|^2}, \ \ \begin{cases}
    & \beta\geq 0 \text{ for } 0<\theta < \frac{1}{4},\\
    & \beta > \frac{3}{2} \text{ for } \theta=0.
\end{cases}    
\end{equation}
Such choice of weight guarantee $w^{-2}(v)\in L^1_v$ and $w(v)\sqrt{\mu}\lesssim 1$.
}

Now we state our main result.

\begin{theorem}\label{thm:linfty}
Assume $\O$ is bounded and smooth. There exists a constant $0<\delta\ll 1$ such that \\
if the initial condition $F_0(x,v)=\mu+\sqrt{\mu}f_0(x,v)\geq 0$ satisfies $\int_{\Omega}\int_{\R^3}\sqrt{\mu}f_0(x,v)\,\dd v\dd x=0$ and
\begin{equation*}
\Vert wf_0\Vert_{L^\infty_{x,v}} < \delta,
\end{equation*}
then there exists a unique solution $F(t,x,v)=\mu+\sqrt{\mu}f(t,x,v)\geq 0$ to the problem \eqref{f_eqn} \\
such that $\int_{\Omega}\int_{\R^3}\sqrt{\mu}f(t,x,v)\,\dd v\dd x\equiv 0$, and the following estimate holds true:
\begin{equation*}
\Vert wf(t)\Vert_{L^\infty_{x,v}} \leq Ce^{-\lambda t}\delta.
\end{equation*}
Here $C>1$, $0<\lambda < 1$ are constants.

\end{theorem}

{\color{black}
\begin{remark}
The linearized BGK operator $\mathcal{L}f = (\mathbf{I}-\mathbf{P})f$ corresponds to the microscopic component of $f$. Here, the first component $f$ serves as a damping factor, while the second component $\mathbf{P}f$ is a compact operator on $L^2_v$.

For the Boltzmann operator $Q(F,F)$, the linearized operator is given by $\mathcal{L}_Q(f) :=- \frac{Q(\mu,\sqrt{\mu}f)+Q(\sqrt{\mu}f,\mu)}{\sqrt{\mu}}$. By the Grad estimate in \cite{R}, this operator can be decomposed into
\begin{align*}
    &  \mathcal{L}_Q f = \nu(v) f - Kf, \ \   Kf = \int_{\mathbb{R}^3} \mathbf{k}(v,u)f(u) \dd u, \ \ \mathbf{k}(v,u) \lesssim \frac{e^{-C|v-u|^2}}{|v-u|}.
\end{align*}

Compared with the linear BGK operator, the damping factor is given by $\nu(v)\sim (1+|v|)^\gamma$. In the case of hard sphere $\gamma = 1$, it provides extra damping. In the case of Maxwell molecule $\gamma = 0$, this coincides with the damping factor in the BGK operator. 

The integral operator $Kf$ is also a compact operator on $L^2_v$. Under a polynomial or exponential weight $w$ from \eqref{weight}, the kernel $\mathbf{k}(v,u)$ enjoys(hard sphere potential)
\begin{align*}
    &   \mathbf{k}(v,u)\frac{w(v)}{w(u)} \lesssim \frac{e^{-C|v-u|^2}}{|v-u|},  \\
    &w(v)Kf =w(v) \int_{\mathbb{R}^3} \mathbf{k}(v,u) f(u) \dd u \lesssim \Vert wf\Vert_{L^\infty_{x,v}} \int_{\mathbb{R}^3} \frac{e^{-C|v-u|^2}}{|v-u|}\dd u \lesssim \Vert wf\Vert_{L^\infty_{x,v}}.
\end{align*}

The linearized BGK operator exhibits a similar but more regular property. Since $\mathbf{P}$ is an $L^2_v$-projection onto its kernel, and given the constraint $\theta<\frac{1}{4}$ in \eqref{weight}, we have:
\begin{align*}
    &   w(v) \mathbf{P}f \lesssim w(v)\sum_{i=0}^4 \chi_i(v)\int_{\mathbb{R}^3} \chi_i(u)f(u) \dd u \lesssim \Vert wf\Vert_{L^\infty_{x,v}} \sum_{i=0}^4 \int_{\mathbb{R}^3} \chi_i(u)w^{-1}(u)\lesssim \Vert wf\Vert_{L^\infty_{x,v}}.
\end{align*}

The integral kernel of the linearized BGK operator does not exhibit a singularity in $|v-u|$. We expect this smoother structure to enhance the regularity of solutions to the BGK equation in boundary value problems, compared with the regularity of the Boltzmann equation studied in \cite{GKTT}. This will be left for future study.

\end{remark}

\begin{remark}\label{remark:nonlinear}
While the linearized BGK operator has a simpler structure than its Boltzmann counterpart, its nonlinear term $\Gamma(f)$ is more complex. The nonlinear Boltzmann operator $\Gamma_Q(f,f) = \frac{Q(\sqrt{\mu}f,\sqrt{\mu}f)}{\sqrt{\mu}}$ has a bilinear form. This structure yields the key weighted estimate: $\Vert \nu^{-1}w\Gamma_Q(f,f)\Vert_{L^\infty_{x,v}}\lesssim \Vert wf\Vert_{L^\infty_{x,v}}^2$ in the $L^2_{x,v}-L^\infty_{x,v}$ argument. Furthermore, the sequential argument and uniqueness proof follow directly, as the bilinearity implies:
\begin{align*}
\Vert \nu^{-1}w[\Gamma_Q(f_1-f_2,f_1)+\Gamma(f_2,f_1-f_2)]\Vert_{L^\infty_{x,v}} \lesssim \Vert w(f_1-f_2)\Vert_{L^\infty_{x,v}}[\Vert wf_1\Vert_{L^\infty_{x,v}} + \Vert wf_2\Vert_{L^\infty_{x,v}}].
\end{align*}

In contrast, the nonlinear BGK operator $\Gamma(f)$ is defined in \eqref{nonlinear_op} via a Taylor expansion, and exhibits a more intricate nonlinearity. Although the last term in \eqref{nonlinear_op} appears to be trilinear, the coefficients $\mathcal{Q}_{ij}$ in \eqref{Q_ij} depend nonlinearly on the macroscopic quantities $(\rho, U, T)$ of $F$, which again depend on the perturbed solution $f$. This nonlinear dependency poses a major challenge for deriving the necessary $L^\infty_{x,v}$ estimate for $\Vert w\Gamma(f)\Vert_{L^\infty_{x,v}}$. We refer to Lemma \ref{lemma:gamma_property} for details.

Moreover, new sequential equations and stability estimate $\Vert w\Gamma(f_1)-\Gamma(f_2)\Vert_{L^\infty_{x,v}}$ need to be introduced to establish the existence, uniqueness, and positivity of the solution. We refer to the stability estimate in Section \ref{sec:gamma_stability} and the new sequential argument in Section \ref{sec:thm_proof}.

Moreover, to establish the solution's existence, uniqueness, and positivity, we must introduce a new iterative scheme and a corresponding stability estimate of the nonlinear operator$\Vert w\Gamma(f_1)-\Gamma(f_2)\Vert_{L^\infty_{x,v}}$. We refer to the stability estimate in Section \ref{sec:gamma_stability} and the sequential argument in Section \ref{sec:thm_proof}.

\end{remark}

As explained in Remark \ref{remark:nonlinear}, our main contribution in this paper can be summarized as follows:
\begin{enumerate}
    \item \textbf{Nonlinear estimates:} We derive the $L^\infty_{x,v}$ estimate for the nonlinear BGK operator $\Gamma(f)$ and establish the a priori estimate in the $L^2$-$L^\infty$ framework, 

    \item \textbf{Well-posedness:} We derive the $L^\infty_{x,v}$ stability estimate for $\Gamma(f)$ and establish a new iterative scheme to prove the existence and uniqueness of the solution.

    \item \textbf{Positivity:} We establish the positivity of the solution through a new sequential argument.
\end{enumerate}

}

\textbf{Outline.} In Section \ref{sec:prelim}, we will derive the expressions of the BGK operator $\mathcal{L}$ and $\Gamma$, and establish their fundamental properties. In Section \ref{sec:l2}, we will derive the $L^2$ estimate for the linear BGK equation by leveraging the coercive property of $\mathcal{L}$. Finally, in Section \ref{sec:linfty}, we conclude Theorem \ref{thm:linfty} by constructing the $L^\infty$ estimate through the method of characteristics and employing an iterative argument for the existence and uniqueness of the solution.

\ \\

\section{Preliminaries}\label{sec:prelim}

\subsection{Derivation of $\mathcal{L}$ and $\Gamma$}

In this section, we derive the explicit expressions of $\mathcal{L}$ and $\Gamma$. 
\begin{lemma}[\cite{Yun}]\label{lemma:operator_property}
(i) The collision frequency $\nu = \rho^\eta T^\omega$ in \eqref{1.1} can be linearized around the global equilibrium state $(\rho,T)=(1,1)$ as
\begin{align}
    & \nu = 1 + \nu_p , \notag\\
    &  \ \nu_p = \sum_i \langle f,\chi_i\rangle \int_0^1 D_{(\rho_\vartheta, \rho_\vartheta U_\vartheta, G_\vartheta)} (\rho_\vartheta^\eta T^\omega_\vartheta)  \dd \vartheta = \sum_{i}\langle f,\chi_i\rangle \int_0^1 Q_i  \dd \vartheta   , \label{nu_p}
\end{align}
where the notation in \eqref{nu_p} is defined as
\begin{align}
    &   Q_i:=  \{ D_{(\rho_\vartheta,\rho_\vartheta U_\vartheta, G_\vartheta)}(\rho_\vartheta^\eta T_\vartheta^\omega)\}_i, \label{Q_i_def} \\
    & \rho_\vartheta = \vartheta \rho + (1-\vartheta) 1 ,\ \rho_\vartheta U_\vartheta = \vartheta \rho U , \notag \\
    & \frac{\rho_\vartheta |U_\vartheta|^2 + 3\rho_\vartheta T_\vartheta}{2} - \frac{3}{2}\rho_\vartheta = \vartheta \Big\{\frac{\rho |U|^2 + 3\rho T}{2} - \frac{3}{2}\rho \Big\}, \notag\\
    & G = \frac{\rho|U|^2+3\rho T}{\sqrt{6}} - \frac{3\rho}{\sqrt{6}}, \ G_\vartheta = \vartheta G.  \label{vartheta_def}
\end{align}

(ii) The local Maxwellian $M(F)$ in \eqref{1.1} can be linearized around $\mu$ as
\begin{align}
    &   M(F) = \mu + \mathbf{P}f\sqrt{\mu} + \sum_{0\leq i,j\leq 4}\langle f,\chi_i\rangle \langle f,\chi_j\rangle \int_0^1 \mathcal{Q}_{ij}  (1-\vartheta) \dd \vartheta  .  \label{M_expansion}
\end{align}
Here $\mathcal{Q}_{ij}$ is defined as
\begin{align}
    &   \mathcal{Q}_{ij} = \{D^2_{(\rho_\vartheta,\rho_\vartheta U_\vartheta,G_\vartheta)}M(\vartheta)\}_{ij}, \label{Q_ij}
\end{align}
\begin{align}
    &   M(\vartheta) = \frac{\rho_\vartheta}{(2\pi T_\vartheta)^{3/2}} e^{-\frac{|v-U_\vartheta|^2}{2T_\vartheta}}.\notag
\end{align}

(iii) Plugging the perturbation $F=\mu+\sqrt{\mu} f$ and the expansion of $M(F),\nu$ given by \eqref{M_expansion}, \eqref{nu_p} into the equation \eqref{1.1}, we derive the expression of $\mathcal{L}(f)$ and $\Gamma(f)$ as
\begin{align}
    &    \mathcal{L}f = (\mathbf{I}-\mathbf{P})f ,  \label{linear_op}  \\
    &    \Gamma(f) = \nu_p \mathbf{P}f -  \nu_p f + \sum_{0\leq i,j \leq 4} \int_0^1 \mathcal{Q}_{ij}(1-\vartheta) \dd \vartheta \langle f,\chi_i\rangle \langle f,\chi_j\rangle \mu^{-1/2} \notag \\
    &+ \nu_p \sum_{0\leq i,j\leq 4} \int_0^1 \mathcal{Q}_{ij}(1-\vartheta) \dd \vartheta\langle f,\chi_i\rangle \langle f,\chi_j\rangle \mu^{-1/2} \notag \\
    & := \Gamma_1(f) + \Gamma_2(f) + \Gamma_3(f) + \Gamma_4(f). \label{nonlinear_op}
\end{align}

\end{lemma}

To fully state the expression of $\Gamma$ in \eqref{nonlinear_op}, we derive the explicit expression of $Q_i$ and $\mathcal{Q}_{ij}$ in the following lemma.
\begin{lemma}\label{lemma:Q_property}
(i) $Q_i$ in \eqref{Q_i_def} takes the following form: 
\begin{align}
    &    Q_i = \frac{P_i(\rho_\vartheta,U_\vartheta,T_\vartheta)}{R_i(\rho_\vartheta,T_\vartheta)}.  \label{Q_i_property}
\end{align}
Here $R_i(\rho_\vartheta,T_\vartheta) = r_{1,i} (\rho_\vartheta)^{r_{2,i}} (T_\vartheta)^{r_{3,i}}$ is monomial, $r_{1,i}>0$ is a positive constant and $r_{2,i},r_{3,i}\geq 0$ are non-negative constants. $P_i$ is a polynomial
\begin{align*}
    & P_i(\rho_\vartheta,U_{\vartheta,1},U_{\vartheta,2},U_{\vartheta,3},T_\vartheta) = \sum_{m\in \mathcal{S}_i} a_m (\rho_\vartheta)^{m_1} (U_{\vartheta,1})^{m_2} (U_{\vartheta,2})^{m_3} (U_{\vartheta,3})^{m_4} (T_\vartheta)^{m_5}.
\end{align*}
Here $a_m$ is a constant, $m=(m_1,\cdots,m_5)$, where $m_i\geq 0$ are non-negative constants, and $\mathcal{S}_i$ corresponds to a collection of finitely many $m$.

(ii) $\mathcal{Q}_{ij}$ in \eqref{Q_ij} takes the following form:
\begin{align}
    &\mathcal{Q}_{ij}:= [D^2_{(\rho_\vartheta,\rho_\vartheta U_\vartheta,G_\vartheta)} M(\vartheta)]_{ij} = \frac{P_{ij}(\rho_\vartheta,v-U_\vartheta,U_\vartheta,T_\vartheta)}{R_{ij}(\rho_\vartheta,T_\vartheta)}M(\vartheta)  .   \label{Q_ij_property}
\end{align}
Here $R_{ij}(\rho_\vartheta,T_\vartheta) = r_{1,ij}(\rho_\vartheta)^{r_{2,ij}} (T_\vartheta)^{r_{3,ij}}$ is a monomial, $r_{1,ij}>0$ is a positive constant and $r_{2,ij},r_{3,ij}\geq 0$ are powers of non-negative integers. $P_{ij}$ is a polynomial
\begin{align*}
    &P_{ij}(\rho_\vartheta,v_1-U_{\vartheta,1},v_2-U_{\vartheta,2},v_3-U_{\vartheta,3},U_{\vartheta,1},U_{\vartheta,2},U_{\vartheta,3},T_\vartheta) \\
& = \sum_{m\in \mathcal{S}_{ij}} a_m (\rho_\vartheta)^{m_1} (v_1-U_{\vartheta,1})^{m_2} (v_2-U_{\vartheta,2})^{m_3} (v_3-U_{\vartheta,3})^{m_4} (U_{\vartheta,1})^{m_5} (U_{\vartheta,2})^{m_6} (U_{\vartheta,3})^{m_7} (T_\vartheta)^{m_8} ,
\end{align*}
here $a_m$ is a constant, $m=(m_1,\cdots,m_8),$ where $m_i\geq 0$ are non-negative integers, and $S_{ij}$ corresponds to a collection of finitely many $m=(m_1,\cdots,m_8)$.

\end{lemma}

\begin{proof}
We can compute the derivative in $Q_i$ as
\begin{align*}
    &   D_{(\rho, \rho U, G)} \rho^\eta T^\omega  = \begin{bmatrix}
        1 & -\frac{U_1}{\rho} & -\frac{U_2}{\rho} & -\frac{U_3}{\rho} & \frac{-3T+|U|^2+3}{3\rho} \\
        0 & \frac{1}{\rho} & 0 & 0 & - \frac{2U_1}{3\rho} \\
        0 & 0 & \frac{1}{\rho} & 0 & -\frac{2U_2}{3\rho} \\
        0 & 0 & 0& \frac{1}{\rho} & -\frac{2U_3}{3\rho} \\
       0 & 0 & 0 & 0 & \frac{\sqrt{\frac{2}{3}}}{\rho}
    \end{bmatrix} \begin{bmatrix}
        \eta \rho^{\eta-1} T^\omega \\
        0 \\
        0 \\
        0 \\
        \omega T^{\omega-1} \rho^\eta
    \end{bmatrix}.
\end{align*}
Here the first matrix corresponds to $(D_{(\rho,\rho U, G)}(\rho,U,T))^{-1}$. This concludes \eqref{Q_i_property}.

Next we compute the derivative in $\mathcal{Q}_{ij}$ as
\begin{align*}
    &   D_{(\rho, \rho U, G)} \frac{\rho}{(2\pi T)^{3/2}} e^{-\frac{|v-U|^2}{2T}}  = \begin{bmatrix}
        1 & -\frac{U_1}{\rho} & -\frac{U_2}{\rho} & -\frac{U_3}{\rho} & \frac{-3T+|U|^2+3}{3\rho} \\
        0 & \frac{1}{\rho} & 0 & 0 & - \frac{2U_1}{3\rho} \\
        0 & 0 & \frac{1}{\rho} & 0 & -\frac{2U_2}{3\rho} \\
        0 & 0 & 0& \frac{1}{\rho} & -\frac{2U_3}{3\rho} \\
       0 & 0 & 0 & 0 & \frac{\sqrt{\frac{2}{3}}}{\rho}
    \end{bmatrix} \begin{bmatrix}
        \frac{1}{\rho} \\
        -\frac{U_1-v_1}{T} \\
        -\frac{U_2-v_2}{T}\\
        -\frac{U_3-v_3}{T} \\
        \frac{|v-U|^2-3T}{2T^2}
    \end{bmatrix} M(F) .
\end{align*}
\hide
\\
    & = \begin{bmatrix}
        1 + \frac{U_1(U_1-v_1)}{T} + \frac{U_2(U_2-v_2)}{T} + \frac{U_3(U_3-v_3)}{T} + \frac{(-3T + |U|^2 +3)(|v-U|^2-3T)}{6T^2} \\
        \frac{v_1}{T} - \frac{U_1|v-U|^2}{3T^2} \\
        \frac{v_2}{T} - \frac{U_2|v-U|^2}{3T^2} \\
        \frac{v_3}{T} - \frac{U_3|v-U|^2}{3T^2} \\
        \sqrt{\frac{2}{3}} \frac{|v-U|^2-3T}{2T^2}
    \end{bmatrix} \frac{1}{(2\pi T)^{3/2}} e^{-\frac{|v-U|^2}{2T}}  \\
    & = \begin{bmatrix}
       -\frac{|v|^2}{2T} + \frac{|v-U|^2}{2T^2} + \frac{|U|^2|v-U|^2}{6T} + \frac{5}{2} -\frac{3}{2T} \\
        \frac{v_1}{T} - \frac{U_1|v-U|^2}{3T^2} \\
        \frac{v_2}{T} - \frac{U_2|v-U|^2}{3T^2} \\
        \frac{v_3}{T} - \frac{U_3|v-U|^2}{3T^2} \\
        \sqrt{\frac{2}{3}} \frac{|v-U|^2-3T}{2T^2}
    \end{bmatrix} \frac{1}{(2\pi T)^{3/2}} e^{-\frac{|v-U|^2}{2T}}.
\end{align*}
\unhide

The second derivative becomes
\begin{align*}
    & D^2_{(\rho,\rho U,G)} M(F) = \begin{bmatrix}
        1 & -\frac{U_1}{\rho} & -\frac{U_2}{\rho} & -\frac{U_3}{\rho} & \frac{-3T+|U|^2+3}{3\rho} \\
        0 & \frac{1}{\rho} & 0 & 0 & - \frac{2U_1}{3\rho} \\
        0 & 0 & \frac{1}{\rho} & 0 & -\frac{2U_2}{3\rho} \\
        0 & 0 & 0& \frac{1}{\rho} & -\frac{2U_3}{3\rho} \\
       0 & 0 & 0 & 0 & \frac{\sqrt{\frac{2}{3}}}{\rho}
    \end{bmatrix}  \\
    & D_{(\rho,U,T)} \left(\begin{bmatrix}
        1 & -\frac{U_1}{\rho} & -\frac{U_2}{\rho} & -\frac{U_3}{\rho} & \frac{-3T+|U|^2+3}{3\rho} \\
        0 & \frac{1}{\rho} & 0 & 0 & - \frac{2U_1}{3\rho} \\
        0 & 0 & \frac{1}{\rho} & 0 & -\frac{2U_2}{3\rho} \\
        0 & 0 & 0& \frac{1}{\rho} & -\frac{2U_3}{3\rho} \\
       0 & 0 & 0 & 0 & \frac{\sqrt{\frac{2}{3}}}{\rho}
    \end{bmatrix} \begin{bmatrix}
        \frac{1}{\rho} \\
        -\frac{U_1-v_1}{T} \\
        -\frac{U_2-v_2}{T}\\
        -\frac{U_3-v_3}{T} \\
        \frac{|v-U|^2-3T}{2T^2}
    \end{bmatrix} M(F) \right).
\end{align*}
This concludes \eqref{Q_ij_property}.

\end{proof}

\begin{lemma}\label{lemma:Pgamma}
The nonlinear operator $\Gamma$ in \eqref{nonlinear_op} satisfies
\begin{align}
    & \mathbf{P}(\Gamma(f)) = 0. \notag
\end{align}
\end{lemma}

\begin{proof}
We use the definition of $\Gamma(f)$ in \eqref{Gamma_def} and have
\begin{align*}
    & \mathbf{P}(\Gamma(f)) = \mathbf{P}\Big(\frac{\nu(M(F)-F)}{\sqrt{\mu}} \Big) - \mathbf{P}((\mathbf{I}-\mathbf{P})f) = \nu\mathbf{P}\Big(\frac{(M(F)-F)}{\sqrt{\mu}} \Big) \\
    & = \sum_{i=0}^4 \nu \chi_i \int_{\mathbb{R}^3} (M(F)-F)\frac{\chi_i}{\sqrt{\mu}} \dd v = 0.
\end{align*}
In the second line, we used {\color{black}$\frac{\chi_0}{\sqrt{\mu}} = 1, \ \frac{\chi_{i}}{\sqrt{\mu}} = v_i, i\in \{1,2,3\}, \  \frac{\chi_4}{\sqrt{\mu}} = \frac{|v|^2-3}{2}$} and the conservation of mass, momentum and energy.

\end{proof}

\subsection{$L^\infty$ estimate of $\Gamma$}\label{sec:gamma_linfty}

As discussed in the introduction, we aim to control the nonlinear operator $\Gamma$ in $L^\infty$ space. In this section, we establish the $L^\infty$ control of $\Gamma$ in Lemma \ref{lemma:gamma_property}. This result will play a crucial role in proving the a priori estimate.

\begin{lemma}\label{lemma:macroscopic_control}
We can control the macroscopic quantities using the $L^\infty$ estimate of $f$ as follows:

If $\Vert wf\Vert_{L^\infty_{x,v}} \lesssim \delta$, then it holds:
\begin{align}
    & \Vert \rho-1, U, T-1 \Vert_{L^\infty_x} \lesssim \delta.  \label{macro_control}
\end{align}
This further leads to
\begin{align}
    & \int_0^1 |Q_i| \dd \vartheta  \lesssim 1,  \label{Q_i_control}\\
    & (1+|v|)^\beta e^{\theta|v|^2}\int_0^1 |\mathcal{Q}_{ij}|(1-\vartheta)\dd \vartheta \mu^{-1/2}\lesssim 1 \ \text{ for }\theta<\frac{1}{4} . \label{Q_ij_control}
\end{align}

\end{lemma}

\begin{proof}
We can estimate the density as
\begin{align*}
    &   |\rho(t,x)-1|  =     \Big|\int_{\mathbb{R}^3} [\mu + \sqrt{\mu} f  ]\dd v - 1 \Big| \leq  \Vert wf\Vert_{L^\infty_{x,v}} \int_{\mathbb{R}^3} \sqrt{\mu} w^{-1} \dd v  \leq  C\delta.
\end{align*}

Then we estimate the momentum as
\begin{align*}
    &     |\rho(t,x)U(t,x)| = \Big|  \int_{\mathbb{R}^3} [\mu+\sqrt{\mu}f] v \dd v   \Big| \leq \Vert wf\Vert_{L^\infty_{x,v}}\int_{\mathbb{R}^3} \sqrt{\mu(v)}  w^{-1}(v) |v| \dd v \leq C\delta.
\end{align*}
Thus
\begin{align*}
    & |U(t,x)| \leq \frac{C\delta}{\inf\{\rho(t,x)\}}\leq \frac{C\delta}{1-C\delta} \leq 2C\delta.
\end{align*}

Last we compute the energy as
\begin{align*}
    &  |3\rho(t,x)T(t,x)-3| =  \Big|\int_{\mathbb{R}^3} (\mu+\sqrt{\mu}f) |v-U(t,x)|^2   \dd v -3 \Big| \\
    &\leq |U(t,x)|^2 + \Vert wf\Vert_{L^\infty_{x,v}} \int_{\mathbb{R}^3} \sqrt{\mu(v)} w^{-1}(v) |v-U(t,x)|^2 \dd v \\
    & \lesssim \delta^2 + \Vert wf\Vert_{L^\infty_{x,v}} \int_{\mathbb{R}^3} \sqrt{\mu(v)} w^{-1}(v) [|v|^2 + |U(t,x)|^2] \dd v\\
    & \leq \delta^2 + C\delta( C+ (2C\delta)^2 ) \lesssim \delta.
\end{align*}
With 
\begin{align*}
    &   T(t,x)-1 = \frac{\rho(t,x)T(t,x)-1}{\rho(t,x)} - \frac{\rho(t,x)-1}{\rho(t,x)},
\end{align*}
we derive that,
\begin{align*}
   | T(t,x)-1| \lesssim \frac{\delta}{1-C\delta} + \frac{\delta}{1-C\delta} \lesssim \delta.
\end{align*}
We conclude \eqref{macro_control}.

Next we prove \eqref{Q_i_control}. Recall the definition of $\rho_\vartheta,U_\vartheta,T_\vartheta$ in Lemma \ref{lemma:operator_property}, from \eqref{macro_control} it is straightforward to verify that for some $C$
\begin{align}
    |\rho_\vartheta-1,U_\vartheta,T_\vartheta-1| \leq C \delta. \label{macro_theta_control}
\end{align}

By the property of $Q_i$ in \eqref{Q_i_property}, we apply \eqref{macro_theta_control} to control the denominator as
\begin{align*}
    & 1\lesssim r_{1,i} (1-C\delta)^{r_{2,i}} (1-C\delta)^{r_{3,i}} \leq  R_i(\rho_\vartheta,T_\vartheta) .
\end{align*}
We control the numerator as
\begin{align*}
    &   P_i(\rho_\vartheta,U_\vartheta,T_\vartheta) \lesssim \sum_{m\in \mathcal{S}_i} |a_m|(1+C\delta)^{m_1}(C\delta)^{m_2+m_3+m_4}(1+C\delta)^{m_5} \lesssim 1.
\end{align*}
This concludes \eqref{Q_i_control}.

Last we prove \eqref{Q_ij_control}. From the property of $\mathcal{Q}_{ij}$ in \eqref{Q_ij_property}, we apply \eqref{macro_theta_control} to control the denominator as
\begin{align*}
    &    1 \lesssim r_{1,ij} (1-C\delta)^{r_{2,ij}} (1-C\delta)^{r_{3,ij}} \leq R_{ij}(\rho_\vartheta,T_\vartheta)
\end{align*}
We control the numerator as
\begin{align*}
    &   P_{ij}(\rho_\vartheta,v_1-U_{\vartheta,1},v_2-U_{\vartheta,2},v_3-U_{\vartheta,3},U_{\vartheta,1},U_{\vartheta,2},U_{\vartheta,3},T_\vartheta) \\
    & \lesssim \sum_{m\in \mathcal{S}_{ij}} |a_m|(1+C\delta)^{m_1} (v_1-U_{\vartheta,1})^{m_2}(v_2-U_{\vartheta,2})^{m_3}(v_3-U_{\vartheta,3})^{m_4}(C\delta)^{m_5+m_6+m_7}(1+C\delta)^{m_9}M(\vartheta) \\
    & \lesssim \sum_{m\in \mathcal{S}_{ij}} |v-U_\vartheta|^{m_2+m_3+m_4} \frac{1+C\delta}{(2\pi (1-C\delta))^{3/2}} e^{-\frac{|v-U_\vartheta|^2}{2(1+C\delta)}} \lesssim e^{-\frac{|v|^2}{2(1+C\delta+C(\theta))}}.
\end{align*}
In the last line, we first bound the polynomial by an exponential as $|v-U_\vartheta|^{m_2+m_3+m_4}\lesssim e^{-c|v-U_\vartheta|^2}$ for some small $c$ that depends on $\theta$ to achieve
\begin{align*}
    & |v-U_\vartheta|^{m_2+m_3+m_4} e^{-\frac{|v-U_\vartheta|^2}{2(1+C\delta)}} \lesssim e^{-\frac{|v-U_\vartheta|^2}{2(1+C\delta+C(\theta))}}.
\end{align*}
Here $C(\theta)$ is a small constant that depends on $\theta$.

Then we bound
\begin{align*}
    & e^{-\frac{|v-U_\vartheta|^2}{2(1+C\delta+C(\theta))}} = e^{\frac{-|v|^2 + 2v\cdot U_\vartheta - |U_\vartheta|^2}{2(1+C\delta+C(\theta))}} \lesssim e^{\frac{-|v|^2  + |v|^2 |U_\vartheta|^2 +2}{2(1+C\delta+C(\theta))}} \lesssim e^{\frac{-(1-C^2\delta^2)|v|^2}{2(1+C\delta+C(\theta))}} \lesssim e^{-\frac{|v|^2}{2(1+C\delta+2C(\theta))}}.
\end{align*}
Since $\delta\ll 1$ and $\theta<\frac{1}{4}$, we can choose $C(\theta)$ and $\delta$ to be small enough such that
\begin{align*}
    &   (1+|v|)^{\beta}e^{\theta|v|^2} e^{-\frac{|v|^2}{2(1+C\delta+2C(\theta))}} e^{|v|^2/4} \lesssim 1.
\end{align*}
Here the inequality does not depend on $\delta$. We conclude the lemma.

\end{proof}

\begin{lemma}\label{lemma:gamma_property}
When $\Vert wf\Vert_{L^\infty_{x,v}} \leq \Vert e^{\lambda t}wf\Vert_{L^\infty_{x,v}} \lesssim \delta$, the following $L^\infty$ control holds for the nonlinear operator given in \eqref{nonlinear_op}:

\begin{align}
    & \Vert w \Gamma_i(f) \Vert_{L^\infty_{x,v}} \lesssim \Vert wf\Vert_{L^\infty_{x,v}}^2,  \ \Vert e^{2\lambda t}w\Gamma_i(f)\Vert_{L^\infty_{x,v}} \lesssim \Vert e^{\lambda t}wf\Vert_{L^\infty_{x,v}}^2, \ i=1,2,3, \label{gamma_est_1}
\end{align}
\begin{align}
    & \Vert w \Gamma_4(f)\Vert_{L^\infty_{x,v}} \lesssim \Vert wf\Vert_{L^\infty_{x,v}}^3, \ \Vert e^{2\lambda t}w\Gamma_4(f)\Vert_{L^\infty_{x,v}} \lesssim \Vert e^{\lambda t}wf\Vert_{L^\infty_{x,v}}^3.   \label{gamma_est_4}
\end{align}

\end{lemma}

\begin{proof}
We first prove \eqref{gamma_est_1}. From \eqref{nonlinear_op}, we apply \eqref{Q_i_control} and compute
\begin{align*}
    & | w \Gamma_1(f) |  \lesssim  |w\mathbf{P}f| \sum_i  \Big| \langle f,\chi_i\rangle \int_0^1 Q_i\dd \vartheta \Big| \lesssim |\sum_i w\chi_i \langle f,\chi_i\rangle| \sum_i |\langle f,\chi_i\rangle| \\
    & \lesssim \Vert wf\Vert_{L^\infty_{x,v}}^2 \sum_i \langle w^{-1},\chi_i\rangle^2 \lesssim \Vert wf\Vert_{L^\infty_{x,v}}^2.
\end{align*}
Here we used $\theta<\frac{1}{4}$ so that $(1+|v|)^\beta e^{\theta|v|^2}\chi_i \lesssim 1$. The second inequality in \eqref{gamma_est_1} follows in the same computation:
\begin{align*}
    &  |e^{2\lambda t} w \Gamma_1(f)| \lesssim |\sum_i w\chi_i \langle e^{\lambda t}f,\chi_i\rangle| \sum_i |\langle e^{\lambda t}f,\chi_i\rangle| \lesssim \Vert e^{\lambda t}wf\Vert_{L^\infty_{x,v}}^2.
\end{align*}

Then for $\Gamma_2(f)$ we apply \eqref{Q_i_control} and have
\begin{align*}
    &  |w\Gamma_2(f)| \lesssim |wf|\sum_{i} \Big|\langle f,\chi_i\rangle \int_0^1 Q_i \dd \vartheta \Big| \lesssim \Vert wf\Vert_{L^\infty_{x,v}}^2 \sum_i \langle w^{-1},\chi_i\rangle \lesssim \Vert wf\Vert_{L^\infty_{x,v}}^2, \\
    & e^{\lambda t}|w\Gamma_2(f)|\lesssim |e^{\lambda t}wf|\sum_i \langle w^{-1},\chi_i\rangle \lesssim \Vert e^{\lambda t}wf\Vert_{L^\infty_{x,v}}^2.
\end{align*}

For $\Gamma_3(f)$, we apply \eqref{Q_ij_control} to have
\begin{align*}
    & |w\Gamma_3(f)| \lesssim  w\sum_{0\leq i,j\leq 4} \Big|\int_0^1 \mathcal{Q}_{ij}(1-\vartheta)\dd \vartheta \Big|  \mu^{-1/2} \Vert wf\Vert_{L^\infty_{x,v}}^2   \lesssim \Vert wf\Vert_{L^\infty_{x,v}}^2 \\
    & e^{2\lambda t}|w\Gamma_3(f)|\lesssim   w\sum_{0\leq i,j\leq 4} \Big|\int_0^1 \mathcal{Q}_{ij}(1-\vartheta)\dd \vartheta \Big|  \mu^{-1/2} \Vert e^{\lambda t} wf\Vert_{L^\infty_{x,v}}^2   \lesssim \Vert e^{\lambda t}wf\Vert_{L^\infty_{x,v}}^2.
\end{align*}
This concludes \eqref{gamma_est_1}.

Next, we prove \eqref{gamma_est_4}. We apply \eqref{Q_i_control} and \eqref{Q_ij_control} to have
\begin{align*}
    &    |w\Gamma_4(f)| \lesssim   \sum_i \langle |f|,\chi_i\rangle \Big|\int_0^1 Q_i \dd \vartheta \Big|  \sum_{0\leq j,k\leq 4}w\Big|\int_0^1 \mathcal{Q}_{ij}(1-\vartheta) \dd \vartheta \Big| \mu^{-1/2}  \langle |f|,\chi_j\rangle \langle |f|,\chi_k\rangle \lesssim \Vert wf\Vert_{L^\infty_{x,v}}^3.
\end{align*}
Similarly, we have
\begin{align*}
    &  |e^{2\lambda t}w\Gamma_4(f)| \lesssim \Vert e^{\lambda t}wf\Vert_{L^\infty_{x,v}}^3.
\end{align*}
This concludes \eqref{gamma_est_4}.

\end{proof}

\subsection{Stability estimate of $\Gamma$}\label{sec:gamma_stability}

To prove the existence and uniqueness of solutions, we will employ a sequential argument. In this section, we derive the $L^\infty$ stability estimate of $\Gamma$ in Lemma \ref{lemma:nonlinear_substraction}.

\begin{lemma}\label{lemma:Q_difference}
Let $F_1 = \mu + \sqrt{\mu}f_1$, $F_2 = \mu + \sqrt{\mu}f_2$ and assume that $\Vert wf_k\Vert_{L^\infty_{x,v}} \lesssim \delta, \ k=1,2.$ We denote $(\rho_k,U_k,T_k)$ as the macroscopic quantities of $F_k$ defined in \eqref{local_maxwellian}, then it holds that
\begin{align}
    &   \int_0^1|Q_i(\rho_{1,\vartheta}, U_{1,\vartheta}, T_{1,\vartheta}) - Q_i(\rho_{2,\vartheta}, U_{2,\vartheta}, T_{2,\vartheta})|\dd \vartheta \lesssim \Vert w(f_1-f_2)\Vert_{L^\infty_{x,v}}, \label{Q_i_difference}
\end{align}
\begin{align}
    & (1+|v|)^\beta e^{\theta|v|^2} \int_0^1 | \mathcal{Q}_{ij}(\rho_{1,\vartheta},v-U_{1,\vartheta},U_{1,\vartheta},T_{1,\vartheta})-\mathcal{Q}_{ij}(\rho_{2,\vartheta},v-U_{2,\vartheta},U_{2,\vartheta},T_{2,\vartheta}) | (1-\vartheta) \dd \vartheta \mu^{-1/2} \notag \\
    &\lesssim \Vert w(f_1-f_2)\Vert_{L^\infty_{x,v}}.  \label{Q_ij_difference}
\end{align}

\end{lemma}

\begin{proof}
From \eqref{macro_control} we have
\begin{align}
    &  \Vert \rho_k-1,U_k,T_k-1\Vert_{L^\infty_{x}} \lesssim \delta, \ k=1,2. \label{macro_k_control}
\end{align}
We compute the difference of the macroscopic quantities $\rho,U,T$ as
\begin{align*}
    & |\rho_1-\rho_2| = \Big|\int_{\mathbb{R}^3} (F_1-F_2) \dd v  \Big| = \Big|\int_{\mathbb{R}^3} (f_1-f_2) \sqrt{\mu} \dd v \Big| \lesssim \Vert w(f_1-f_2)\Vert_{L^\infty_{x,v}}, \\
    &  \\
    & |U_1-U_2| = \Big| \frac{1}{\rho_1} \int_{\mathbb{R}^3} F_1 v \dd v - \frac{1}{\rho_2} \int_{\mathbb{R}^3} F_2 v \dd v  \Big|   \\
    &= \Big|\frac{1}{\rho_1} \int_{\mathbb{R}^3} (F_1-F_2)v \dd v + \int_{\mathbb{R}^3} F_2 v \dd v \Big(\frac{1}{\rho_1}-\frac{1}{\rho_2} \Big)   \Big| \\ 
    & =  \Big|\frac{1}{\rho_1} \int_{\mathbb{R}^3} (f_1-f_2)v\sqrt{\mu} \dd v + \int_{\mathbb{R}^3} f_2 v\sqrt{\mu} \dd v \Big(\frac{1}{\rho_1}-\frac{1}{\rho_2} \Big)   \Big|   \\
    & \lesssim \Vert w(f_1-f_2)\Vert_{L^\infty_{x,v}} + \Vert wf_2\Vert_{L^\infty_{x,v}} \frac{|\rho_1-\rho_2|}{\rho_1\rho_2} \lesssim \Vert w(f_1-f_2)\Vert_{L^\infty_{x,v}}, \\
    &  \\
    & |T_1-T_2| = \frac{1}{3} \Big|\frac{1}{\rho_1} \int_{\mathbb{R}^3} F_1 |v-U_1|^2  \dd v - \frac{1}{\rho_2} \int_{\mathbb{R}^3} F_2 |v-U_2|^2 \dd v  \Big| \\
    & = \frac{1}{3} \Big|\frac{1}{\rho_1} \int_{\mathbb{R}^3} F_1|v-U_1|^2 - F_2 |v-U_2|^2 \dd v + \int_{\mathbb{R}^3} F_2|v-U_2|^2 \dd v \Big(\frac{1}{\rho_1}-\frac{1}{\rho_2} \Big)    \Big| \\
    & \lesssim \int_{\mathbb{R}^3} |F_1-F_2||v-U_1|^2 \dd v + \int_{\mathbb{R}^3} F_2 \Big| |v-U_1|^2 - |v-U_2|^2 \Big| \dd v + |\rho_1-\rho_2| \int_{\mathbb{R}^3} f_2|v-U_2|^2\sqrt{\mu} \dd v \\
    & \lesssim \int_{\mathbb{R}^3} |f_1-f_2||v-U_1|^2 \sqrt{\mu} \dd v + \int_{\mathbb{R}^3} f_2 [(|U_1|^2 - |U_2|^2) + (U_1-U_2)\cdot v ] \sqrt{\mu} \dd v + \Vert w(f_1-f_2)\Vert_{L^\infty_{x,v}} \\
    & \lesssim \Vert w(f_1-f_2)\Vert_{L^\infty_{x,v}} + |U_1-U_2| \lesssim \Vert w(f_1-f_2)\Vert_{L^\infty_{x,v}}.
\end{align*}

These estimates imply the following control: for integer power $m\geq 1$, due to \eqref{macro_k_control}, we have
\begin{align}
    &   |(\rho_1^m-\rho_2^m, U_1^m - U_2^m, T_1^m-T_2^m)|\notag \\
    &\lesssim |(\rho_1-\rho_2,U_1-U_2,T_1-T_2)| |\rho_1^{m-1} + \rho_2^{m-1} + U_1^{m-1} + U_2^{m-1} + T_1^{m-1}+ T_2^{m-1}| \notag\\
    &\lesssim |(\rho_1-\rho_2,U_1-U_2,T_1-T_2)| \lesssim \Vert w(f_1-f_2)\Vert_{L^\infty_{x,v}}. \label{difference_integer}
\end{align}

For positive power $m>0$, 
\begin{align}
    &    |(\rho_1^m-\rho_2^m,T^m_1-T^m_2)|  \lesssim |\rho_1-\rho_2||\rho_1^{m-1}+\rho_2^{m-1}+T_1^{m-1} + T_2^{m-1}| \notag \\
    & \lesssim |(\rho_1^m-\rho_2^m,T^m_1-T^m_2)| \lesssim \Vert w(f_1-f_2)\Vert_{L^\infty_{x,v}}.   \label{difference_positive}
\end{align}
Here $\rho_k^{m-1},T_k^{m-1}\lesssim 1$ for finite $m$ due to \eqref{macro_k_control}.

For positive power $m>0$, again due to \eqref{macro_k_control},
\begin{align}
    & \Big|\Big(\frac{1}{\rho_1^m}-\frac{1}{\rho_2^m}, \frac{1}{T_1^m}-\frac{1}{T_2^m}\Big) \Big| = \Big|\Big(\frac{\rho_1^m - \rho_2^m}{\rho_1^m \rho_2^m}, \frac{T_1^m - T_2^m}{T_1^m T_2^m}  \Big) \Big| \notag\\
    & \lesssim |(\rho_1^m-\rho_2^m, T_1^m-T_2^m)| \lesssim \Vert w(f_1-f_2)\Vert_{L^\infty_{x,v}}.    \label{difference_negative}
\end{align}

We compute the difference of $\rho_{k,\vartheta},U_{k,\vartheta},T_{k,\vartheta}$ using the definition in \eqref{vartheta_def} and the computation \eqref{difference_integer}, \eqref{difference_positive}, \eqref{difference_negative}:
\begin{align*}
    & |\rho_{1,\vartheta}-\rho_{2,\vartheta},U_{1,\vartheta}-U_{2,\vartheta}| \\
    & = \Big|\vartheta(\rho_1-\rho_2), \frac{\vartheta(\rho_1U_1-\rho_2U_2) - \vartheta U_{2,\vartheta}(\rho_{1}-\rho_{2})}{\rho_{1,\vartheta}}\Big| \lesssim \Vert w(f_1-f_2)\Vert_{L^\infty_{x,v}}, \\
    & |T_{1,\vartheta} - T_{2,\vartheta}| = \frac{2}{3\rho_{1,\vartheta}}\Big|-\frac{\rho_{1,\vartheta}|U_{1,\vartheta}|^2 - \rho_{2,\vartheta}|U_{2,\vartheta}|^2}{2} -\frac{3}{2}T_{2,\vartheta}(\rho_{1,\vartheta}-\rho_{2,\vartheta}) \\
    & + \vartheta \Big\{ \frac{\rho_1 |U_1|^2 + 3\rho_1 T_1}{2} - \frac{3}{2}\rho_1 - \frac{\rho_2 |U_2|^2 + 3\rho_2 T_2}{2} + \frac{3}{2}\rho_2 \Big\}  \Big| \lesssim \Vert w(f_1-f_2)\Vert_{L^\infty_{x,v}}.
\end{align*}

It is straightforward to verify that we can achieve the same estimate for $\rho_{k,\vartheta},U_{k,\vartheta},T_{k,\vartheta}$ as \eqref{difference_integer}, \eqref{difference_positive}, \eqref{difference_negative}:
\begin{align}
    &  |(\rho_{1,\vartheta}^m-\rho_{2,\vartheta}^m, U_{1,\vartheta}^m - U_{2,\vartheta}^m, T_{1,\vartheta}^m-T_{2,\vartheta}^m)| \lesssim \Vert w(f_1-f_2)\Vert_{L^\infty_{x,v}}, \ m \text{ is a positive integer.} \notag \\
    & |(\rho_{1,\vartheta}^m-\rho_{2,\vartheta}^m,T^m_{1,\vartheta}-T^m_{2,\vartheta})| \lesssim \Vert w(f_1-f_2)\Vert_{L^\infty_{x,v}}, \ m \text{ is a positive constant.}   \label{power_m_subtract} \\
    & \Big|\Big(\frac{1}{\rho_{1,\vartheta}^m}-\frac{1}{\rho_{2,\vartheta}^m}, \frac{1}{T_{1,\vartheta}^m}-\frac{1}{T_{2,\vartheta}^m}\Big) \Big| \lesssim \Vert w(f_1-f_2)\Vert_{L^\infty_{x,v}}, \ m \text{ is a positive constant.} \notag
\end{align}

From \eqref{Q_i_property} and \eqref{Q_ij_property}, the denominator of $Q_i$ and $\mathcal{Q}_{ij}$ contain monomial of $\rho_\vartheta, T_\vartheta$, while the numerator contain polynomial of $\rho_\vartheta, U_\vartheta, T_\vartheta,v-U_\vartheta$ with integer powers. Then we can apply the computation \eqref{power_m_subtract} for the subtraction in \eqref{Q_i_difference} and \eqref{Q_ij_difference}. Then the lemma follows by a rather tedious but straightforward computation.

\end{proof}

\begin{lemma}\label{lemma:nonlinear_substraction}
Given $f_1$ and $f_2$ such that $\Vert e^{\lambda t}wf_1\Vert_{L^\infty_{x,v}} + \Vert e^{\lambda t}wf_2\Vert_{L^\infty_{x,v}}\lesssim\delta$, it holds that
\begin{align}
    &  \Vert e^{2\lambda t}w(\Gamma(f_1)-\Gamma(f_2))\Vert_{L^\infty_{x,v}} \lesssim \delta \Vert e^{\lambda t}w(f_1-f_2) \Vert_{L^\infty_{x,v}}. \label{gamma_diff}
\end{align}

\end{lemma}

\begin{proof}
We will derive the lemma by estimating every term in \eqref{nonlinear_op}.

We start with $\Gamma_1(f) = \nu_p(f) \mathbf{P}f_1$. From the definition of $\nu_p$ in \eqref{nu_p}, we compute that
\begin{align*}
    &     \nu_p(f_1) \mathbf{P}f_1 - \nu_p(f_2) \mathbf{P}f_2 = (\nu_p(f_1)-\nu_p(f_2))\mathbf{P}f_1 + \nu_p(f_2) \mathbf{P}(f_1-f_2).
\end{align*}
The second term is bounded using \eqref{nu_p} and \eqref{Q_i_control}:
\begin{align*}
    & e^{2\lambda t}|w\nu_p(f_2) \mathbf{P}(f_1-f_2)| \lesssim \Vert e^{\lambda t}w(f_1-f_2)\Vert_{L^\infty_{x,v}} \Vert e^{\lambda t}wf_2\Vert_{L^\infty_{x,v}} \lesssim \delta \Vert e^{\lambda t}w(f_1-f_2)\Vert_{L^\infty_{x,v}}
\end{align*}
For the first term, we use the property of $Q_i$ \eqref{Q_i_property} to have
\begin{align}
    &  e^{\lambda t}|\nu_p(f_1)-\nu_p(f_2)| \notag \\
    &\lesssim \sum_{i}\langle e^{\lambda t}|f_1-f_2|,\chi_i\rangle \int_0^1 |Q_i(f_1)| \dd \vartheta + \sum_i \langle e^{\lambda t}|f_2|,\chi_i\rangle \int_0^1 |Q_i(f_1)-Q_i(f_2)| \dd \vartheta \notag\\
    & \lesssim \Vert e^{\lambda t} w(f_1-f_2)\Vert_{L^\infty_{x,v}} + \Vert e^{\lambda t}wf_2\Vert_{L^\infty_{x,v}} \Vert w(f_1-f_2)\Vert_{L^\infty_{x,v}} \lesssim \Vert e^{\lambda t}w(f_1-f_2)\Vert_{L^\infty_{x,v}}. \label{nuf_1-nuf_2}
\end{align}
Here we have used \eqref{Q_i_control} and \eqref{Q_i_difference}. 

This leads to
\begin{align*}
    &  |e^{2\lambda t}w(\nu_p(f_1)-\nu_p(f_2))\mathbf{P}f_1| \lesssim \Vert e^{\lambda t}wf_1\Vert_{L^\infty_{x,v}} \Vert e^{\lambda t}w(f_1-f_2)\Vert_{L^\infty_{x,v}} \lesssim \delta \Vert e^{\lambda t}w(f_1-f_2)\Vert_{L^\infty_{x,v}}.
\end{align*}
Thus \eqref{gamma_diff} holds for $\Gamma_1$.

Next we estimate $\Gamma_2(f) = -\nu_p(f) f$. We compute that
\begin{align*}
    &  \nu_p(f_1)f_1 - \nu_p(f_2)f_2 = (\nu_p(f_1)-\nu_p(f_2))f_1 + \nu_p(f_2)(f_1-f_2).
\end{align*}
The second term is bounded using \eqref{nu_p} and \eqref{Q_i_control}:
\begin{align*}
    &  |e^{2\lambda t}w(f_1-f_2)\nu_p(f_2)| \lesssim \Vert e^{\lambda t}w(f_1-f_2)\Vert_{L^\infty_{x,v}} \Vert e^{\lambda t}wf_2\Vert_{L^\infty_{x,v}} \lesssim \delta \Vert e^{\lambda t}w(f_1-f_2)\Vert_{L^\infty_{x,v}}.
\end{align*}
For the first term, using the same computation as \eqref{nuf_1-nuf_2}, we obtain
\begin{align*}
    &   |e^{2\lambda t} w(\nu_p(f_1)-\nu_p(f_2))f_1|\lesssim  \Vert e^{\lambda t}w(f_1-f_2)\Vert_{L^\infty_{x,v}} \Vert e^{\lambda t}wf_1\Vert_{L^\infty_{x,v}} \lesssim \delta \Vert e^{\lambda t}w(f_1-f_2)\Vert_{L^\infty_{x,v}}.
\end{align*}
Thus \eqref{gamma_diff} holds for $\Gamma_2$.

Next we estimate $\Gamma_3(f) = \sum_{0\leq i,j\leq 4} \int_0^1 \mathcal{Q}_{ij}(f)(1-\vartheta)\dd \vartheta \langle f,\chi_i\rangle \langle f,\chi_j\rangle \mu^{-1/2}$. We compute that
\begin{align}
    &   e^{2\lambda t}|w\Gamma_3(f_1)- w\Gamma_3(f_2)| \leq w\sum_{0\leq i,j\leq 4}\int_0^1 [\mathcal{Q}_{ij}(f_1) - \mathcal{Q}_{ij}(f_2)](1-\vartheta) \dd \vartheta \langle e^{\lambda t} f_1,\chi_i\rangle \langle e^{\lambda t}f_1,\chi_j\rangle \mu^{-1/2} \notag\\
    & + w\sum_{0\leq i,j\leq 4}\int_0^1 |\mathcal{Q}_{ij}(f_2)|(1-\vartheta)\dd \vartheta [\langle e^{\lambda t}(f_1-f_2),\chi_i\rangle \langle e^{\lambda t} f_1,\chi_j\rangle + \langle e^{\lambda t} f_2,\chi_i\rangle \langle e^{\lambda t}(f_1-f_2),\chi_j\rangle      ] \mu^{1/2} \notag\\
    & \lesssim \Vert e^{\lambda t} wf_1\Vert_{L^\infty_{x,v}}^2 \Vert w(f_1-f_2)\Vert_{L^\infty_{x,v}} + \Vert e^{\lambda t}w(f_1-f_2)\Vert_{L^\infty_{x,v}} [\Vert e^{\lambda t} wf_1\Vert_{L^\infty_{x,v}} + \Vert e^{\lambda t} wf_2\Vert_{L^\infty_{x,v}}] \notag \\
    &\lesssim \delta \Vert e^{\lambda t} w(f_1-f_2)\Vert_{L^\infty_{x,v}}. \label{gamma_3_difference}
\end{align}
Here we have applied \eqref{Q_ij_control} and \eqref{Q_ij_difference}. Thus \eqref{gamma_diff} holds for $\Gamma_3$.

Last we estimate $\Gamma_4(f)=\nu_p(f)\sum_{0\leq i,j\leq 4} \int_0^1 \mathcal{Q}_{ij}(f)(1-\vartheta)\dd \vartheta \langle f,\chi_i\rangle \langle f,\chi_j\rangle \mu^{-1/2}$. We compute that
\begin{align*}
    &    e^{2\lambda t}|w\Gamma_4(f_1)-w\Gamma_4(f_2)| \leq e^{2\lambda t}|\nu_p(f_1)-\nu_p(f_2)| w \int_0^1 |\mathcal{Q}_{ij}(f_1)|(1-\vartheta)\dd \vartheta \langle f_1,\chi_i\rangle \langle f_1,\chi_j\rangle \mu^{-1/2} \\
    &  + |\nu_p(f_2)| e^{\lambda t}|w\Gamma_3(f_1) - w\Gamma_3(f_2)| \\
    & \lesssim \Vert e^{\lambda t} w(f_1-f_2)\Vert_{L^\infty_{x,v}} \Vert e^{\lambda t} wf_1\Vert_{L^\infty_{x,v}}^2 + \delta \Vert e^{\lambda t} wf_2\Vert_{L^\infty_{x,v}} \Vert e^{\lambda t} w(f_1-f_2)\Vert_{L^\infty_{x,v}} \lesssim \delta \Vert e^{\lambda t} w(f_1-f_2)\Vert_{L^\infty_{x,v}}.
\end{align*}
In the RHS of the first line, we have applied \eqref{nuf_1-nuf_2}, \eqref{Q_ij_control}. In the last line, we have applied \eqref{nu_p}, \eqref{Q_i_control}, and \eqref{gamma_3_difference}. 

We conclude the lemma.

\end{proof}

\ \\

\section{$L^2$ dissipation estimate}\label{sec:l2}

In this section, we consider the solution to the linearized BGK equation 
\begin{align}
   \p_t f + v\cdot \nabla_x f + \mathcal{L}f & =   g,    
   \label{linear_f}
\end{align}
with the source term $g=g(t,x,v)$ satisfying 
\begin{equation}\label{assumption_g}
\mathbf{P}g=0.   
\end{equation}
The boundary condition of $f$ is given by 
\begin{align}
    & f(t,x,v)|_{\gamma_-} = c_\mu\sqrt{\mu(v)} \int_{n(x)\cdot u>0} (n(x)\cdot u) f(t,x,u)\sqrt{\mu(u)} \dd u. \label{diffuse_f}
\end{align}
We denote
\begin{align*}
    & P_\gamma f := c_\mu\sqrt{\mu(v)}\int_{n(x)\cdot u>0} (n(x)\cdot u) f(t,x,u) \sqrt{\mu(u)}\dd u.
\end{align*}

We will prove the following $L^2$-dissipation result.

\begin{proposition}\label{prop:l2}
Let $\O$ be an arbitrary bounded and $C^1$ domain. There exists $0<\lambda\ll 1$ such that if the initial data $f_0$ and source data $g$ satisfy \eqref{assumption_g} and
\[\Vert f_0\Vert_{L^2_{x,v}}^2 + \int_0^t \Vert e^{\lambda s}g(s)\Vert_{L^2_{x,v}}^2\,\dd s < \infty,\]
then there exists a unique solution to the problem
\begin{align}\label{intial_value_problem}
\begin{cases}
    &\p_t f + v\cdot \nabla_x f + \mathcal{L}f =   g  , \\
    &   f(0,x,v) = f_0(x,v), \ f|_{\gamma_-} = P_\gamma f.
\end{cases}
\end{align}
Moreover, it holds that
\begin{equation}\label{l2_decay}
\Vert f(t)\Vert_{L^2_{x,v}}^2 \lesssim e^{-2\lambda t} \Big\{\Vert f_0\Vert_{L^2_{x,v}}^2 + \int_0^t  \Vert e^{\lambda s}g(s)\Vert_{L^2_{x,v}}^2 \,\dd s  \Big\},\quad\forall\,t\geq 0.
\end{equation}

\end{proposition}

To prove Proposition \ref{prop:l2}, we need to have the following $L^2$ dissipation estimate of the macroscopic quantities. Denote 
\begin{align*}
    & |f|_{2,+}^2 : = \int_{\p\O}\int_{n(x)\cdot v>0} |f(t,x,v)|^2 (n(x)\cdot v) \dd v \dd S_x, \ \dd S_x \text{ is the surface integral.}
\end{align*}
\begin{lemma}\label{lemma:L2}
Suppose $f$ solves the following equation,
\begin{equation}\label{linear_f_lambda}
\p_t f +v\cdot \nabla_x f + \mathcal{L}f  =  g,
\end{equation}
with boundary condition \eqref{diffuse_f}. Here $g$ does not need to satisfy the condition \eqref{assumption_g}. It holds that
\begin{align*}
   \int_0^t \Vert \mathbf{P}f(s)\Vert_{L^2_{x,v}}^2 \,\dd s & \lesssim G(t)-G(0) + \int_0^t \Vert (\mathbf{I}-\mathbf{P})f\Vert_{L^2_{x,v}}^2 \,\dd s + \int_0^t \Vert g(s)\Vert_{L^2_{x,v}}^2 \,\dd s\notag \\
   &  \ \ \ + \int_0^t  |(I-P_\gamma) f(s)|_{2,+}^2 \,\dd s,
\end{align*}
where $G(t)$ is a functional of $f(t,x,v)$ such that $|G(t)|\lesssim \Vert f(t)\Vert_{L^2_{x,v}}^2$ holds true for any $t\geq 0$.
\end{lemma}
The proof of the macroscopic dissipation estimate is standard. For completeness, we refer to the proof provided in the appendix.

\begin{proof}[\textbf{Proof of Proposition \ref{prop:l2}}]
We prove the decay estimate~\eqref{l2_decay}. Multiplying~\eqref{intial_value_problem} with $e^{\lambda t}$ we get
\begin{equation}\label{f_lambda_t}
[\p_t  + v\cdot \nabla_x + \mathcal{L}](e^{\lambda t}f) =    \lambda e^{\lambda t}f + e^{\lambda t}g. 
\end{equation}
Applying Green's identity to~\eqref{f_lambda_t}, we have
\begin{align}
&  \Vert e^{\lambda t}f(t)\Vert_{L^2_{x,v}}^2 + \int_0^t \Vert (\mathbf{I}-\mathbf{P})e^{\lambda s}f(s)\Vert_{L^2_{x,v}}^2  \dd s+ \int_0^t  |(I-P_\gamma)e^{\lambda s}f(s)|_{2,+}^2    \dd s \notag\\
& \lesssim  \lambda \int_0^t \Vert e^{\lambda s} f(s)\Vert_{L^2_{x,v}}^2 \dd s + \Vert f(0)\Vert_{L^2_{x,v}}^2 + \int_0^t  \Vert e^{\lambda s} g(s)\Vert_{L^2_{x,v}}^2 \dd s. \label{green_lambda}
\end{align}
Here, we have applied $\mathbf{P}(e^{\lambda t}g)=0$. We have also applied the following coercive property of the diffuse reflection boundary condition:
{\color{black}
\begin{align*}
    & \int_{\p \O}\int_{\mathbb{R}^3} |f(t,x,v)|^2 (n(x)\cdot v) \dd v \dd S_x \\
    &= \int_{\p\O}\int_{n(x)\cdot v>0}|(I-P_\gamma)f+P_\gamma f|^2(n(x)\cdot v) \dd v \dd S_x + \int_{\p\O}\int_{n(x)\cdot v<0} |P_\gamma f|^2 (n(x)\cdot v)\dd v \dd S_x\\
    & = \int_{\p\O} \int_{n(x)\cdot v>0} |(I-P_\gamma)f|^2 (n(x)\cdot v)\dd v  \dd S_x + 2 \int_{\p\O}\int_{n(x)\cdot v>0} fP_\gamma f (n(x)\cdot v)\dd v \dd S_x    \\
    & - 2\int_{\p\O}\int_{n(x)\cdot v>0} |P_{\gamma}f|^2 (n(x)\cdot v)\dd v \dd S_x + \int_{\p\O}\int_{\mathbb{R}^3} |P_{\gamma}f|^2 (n(x)\cdot v)\dd v \dd S_x   \\
    & = |(I-P_\gamma)f|_{2,+}^2 + \Big(\int_{\p\O}\int_{n(x)\cdot u>0}(n(x)\cdot u)f(u)\sqrt{\mu(u)}\dd u\dd S_x\Big)^2\\
    &\times [(2-2)\int_{n(x)\cdot v>0}c_\mu \mu(v)(n(x)\cdot v)\dd v + \int_{\mathbb{R}^3}c_\mu \mu(v)(n(x)\cdot v)\dd v]  \\
    & = |(I-P_\gamma)f|_{2,+}^2.
\end{align*}
}
Next, we apply Lemma \ref{lemma:L2} to~\eqref{f_lambda_t}, then we obtain
\begin{align}
  \int_0^t \Vert e^{\lambda s}\mathbf{P}f(s)\Vert_{L^2_{x,v}}^2 \dd s  &  \lesssim G(t)-G(0) + \int_0^t \Vert e^{\lambda s}(\mathbf{I}-\mathbf{P})f(s)\Vert_{L^2_{x,v}}^2 \dd s + \int_0^t \Vert \lambda e^{\lambda s}f\Vert^2_{L^2_{x,v}} \dd s \notag \\
  & + \int_0^t \Vert e^{\lambda s} g(s)\Vert_{L^2_{x,v}}^2 \dd s + \int_0^t |e^{\lambda s}    (I-P_\gamma) f(s)|_{2,+}^2   \dd s,   \label{pf_lambda}
\end{align}
where $|G(t)|\lesssim \Vert e^{\lambda t} f(t)\Vert_{L^2_{x,v}}^2$. Therefore, multiplying \eqref{pf_lambda} by a small constant $\e$ and adding the resultant to~\eqref{green_lambda}, we obtain that for some $C>0$,
\begin{align*}
 & (1-C\e) \Vert e^{\lambda t}f(t)\Vert_{L^2_{x,v}}^2  +  \Big\{(1-C\e)\int_0^t \Vert e^{\lambda s}(\mathbf{I}-\mathbf{P})f(s)\Vert_{L^2_{x,v}}^2 \dd s + \e\int_0^t\Vert e^{\lambda s}\mathbf{P}f(s)\Vert_{L^2_{x,v}}^2  \dd s\Big\} \\
 &+ (1-C\e) \int_0^t |e^{\lambda s}    (I-P_\gamma) f(s)|_{2,+}^2   \dd s \\
 & \leq  C(\lambda+\lambda^2) \int_0^t \Vert e^{\lambda s} f(s)\Vert_{L^2_{x,v}}^2 \dd s   + C\Vert f(0)\Vert_{L^2_{x,v}}^2   +  \int_0^t \Vert e^{\lambda s} g(s)\Vert_{L^2_{x,v}}^2 \dd s  .
\end{align*}
Since $\Vert e^{\lambda s}(\mathbf{I}-\mathbf{P})f(s)\Vert_{L^2_{x,v}}^2 + \Vert e^{\lambda s}\mathbf{P}f(s)\Vert_{L^2_{x,v}}^2 = \Vert e^{\lambda s}f(s)\Vert_{L^2_{x,v}}^2$, we further obtain that for $\e \ll 1$:
\begin{align*}
 &  \Vert e^{\lambda t}f(t)\Vert_{L^2_{x,v}}^2  +  \e \int_0^t \Vert e^{\lambda s}f(s)\Vert_{L^2_{x,v}}^2 \dd s \\
 & \leq  C(\lambda+\lambda^2) \int_0^t \Vert e^{\lambda s} f(s)\Vert_{L^2_{x,v}}^2 \dd s   + C\Vert f(0)\Vert_{L^2_{x,v}}^2   +  \int_0^t \Vert e^{\lambda s} g(s)\Vert_{L^2_{x,v}}^2 \dd s  .
\end{align*}
Last we let $\lambda \ll 1$ be such that $C(\lambda+\lambda^2) \leq \e$, then the above estimate gives the desired decay estimate \eqref{l2_decay}. We conclude the proof of Proposition \ref{prop:l2}.
\end{proof}

\ \\

\section{$L^\infty$ estimate by method of characteristic}\label{sec:linfty}
In this section, we are devoted to the proof of Theorem \ref{thm:linfty}. We will control the nonlinear operator $\Gamma(f)$ in \eqref{f_eqn} using $L^\infty$ norm. For this purpose, we start with the $L^\infty$ estimate of the linear problem~\eqref{linear_f} in the following proposition.

\begin{proposition}\label{prop:linfty}
Suppose the initial condition and source term in \eqref{linear_f} satisfy
\begin{align}
    &\Vert wf_0\Vert_{L^\infty_{x,v}} < \infty, \    \sup_{0\leq s\leq t}e^{\lambda s}\Vert wg(s)\Vert_{L^\infty_{x,v}} < \infty, \ \mathbf{P}(g) = 0, \label{linfty_condition}
\end{align}
then there exists $C>0$ such that the unique solution in Proposition \ref{prop:l2} satisfies
\begin{align*}
 \Vert wf(t)\Vert_{L^\infty_{x,v}}  &   \leq  Ce^{-\lambda t} \Big\{\Vert wf_0\Vert_{L^\infty_{x,v}} + \sup_{0\leq s\leq t} e^{\lambda s}\Vert w g(s) \Vert_{L^\infty_{x,v}} \Big\},
\end{align*} 
for any $t\geq 0$.
\end{proposition}

We will derive the $L^\infty$ control using the method of characteristics. We use standard notations for the backward exit time and backward exit position:
\begin{equation*}
\tb(x,v) :    = \sup\{s\geq 0, x-sv \in \O\} , \ \xb(x,v)  :  = x - \tb(x,v)v.
\end{equation*}

We denote $t_0=T_0$, a fixed starting time. Since the backward trajectory may have multiple interactions with the boundary, we define the following stochastic cycle:
\begin{definition}\label{definition: sto cycle}
We define a stochastic cycles as $(x_0,v_0)= (x,v) \in \bar{\O} \times \R^3$ and inductively
\begin{align}
&x_1:= \xb(x,v), \   v_1 \in \{v_1\in \mathbb{R}^3:n(x_1)\cdot v_1>0\} , \notag\\
&  v_{k}\in \mathcal{V}_k:= \{v_{k}\in \mathbb{R}^3:n(x_k)\cdot v_k>0\}, \ \ \text{for} \  k \geq 1,
\label{xi}
\\
 &x_{k+1} := \xb(x_k, v_k) , \ \tb^{k}:= \tb(x_k,v_k) \ \ \text{for} \  n(x_k) \cdot v_k\geq 0  , \notag \\
 &t_{k} =  t_0-  \{ \tb + \tb^1 + \cdots + \tb^{k-1}\},  \ \ \text{for} \  k \geq 1. \notag
\end{align}

\end{definition}

With the stochastic cycles defined, suppose $f$ satisfies the linear equation \eqref{linear_f}, with $\mathcal{L}f$ defined in \eqref{linear_op} in Lemma \ref{lemma:operator_property}, we apply the method of characteristic to get
\begin{align}
  &w(v) f(T_0,x,v)  \notag \\
  & =  \mathbf{1}_{t_1\leq 0} w(v) e^{-  T_0} f(0,x-tv, v) \label{initial}\\
  & + \mathbf{1}_{t_1 \leq 0} \int_0^{T_0} e^{- (T_0-s)}  w(v) \sum_{i=0}^4 \chi_i(v) \int_{\mathbb{R}^3}   f(s,x-(t-s)v,u) \chi_i(u) \dd u \dd s \label{K_0}\\
  & +\mathbf{1}_{t_1>0} \int^{T_0}_{t_1} e^{- (T_0 - s)}  w(v) \sum_{i=0}^4 \chi_i(v) \int_{\mathbb{R}^3}  f(s,x-(t-s)v,u) \chi_i(u) \dd u \dd s  \label{K_1} \\
  &  + \mathbf{1}_{t_1 \leq 0} \int_0^{T_0} e^{-(T_0-s)}  w(v)  g(s,x-(t-s)v,v) \dd s  \label{g_0} \\
  &  +\mathbf{1}_{t_1>0} \int^{T_0}_{t_1} e^{- (T_0 - s)} w(v) g(s,x-(t-s)v,v) \dd s   \label{g_1} \\
  & + \mathbf{1}_{t_1>0} e^{- (T_0 - t_1)} w(v)f(t_1,x_1,v). \label{f_bdr}
\end{align}
The boundary term \eqref{f_bdr} is represented using the diffuse reflection boundary condition \eqref{f_eqn} and the stochastic cycle Definition \ref{definition: sto cycle}:
\begin{align*}
    & f(t_1,x_1,v) = c_\mu\sqrt{ \mu(v)} \int_{n(x_1)\cdot v_1>0} f(t_1,x_1,v_1)\sqrt{\mu(v_1)}|n(x_1)\cdot v_1|\dd v_1.
\end{align*}
Applying the method of characteristic again to $f(t_1,x_1,v_1)$, with the stochastic cycle defined in Definition \ref{definition: sto cycle}, it is standard to derive the following bound for the boundary term \eqref{f_bdr}:
\begin{align}
   & |\eqref{f_bdr}|\leq   e^{- (T_0-t_1)} w(v)  \notag\\
   & \times \int_{\prod_{j=1}^{k-1}\mathcal{V}_j} \bigg\{ \sum_{i=1}^{k-1}\mathbf{1}_{t_{i+1}\leq 0 < t_i} e^{- t_i} w(v_i)|f(0, x_i - t_i v_i, v_i)| \dd \Sigma_i     \label{bdr_initial}\\
  & + \mathbf{1}_{t_{k}>0}  w(v_{k-1})|f(t_{k},x_k,v_{k-1})| \dd \Sigma_{k-1} \label{bdr_tk}\\
  & + \sum_{i=1}^{k-1} \mathbf{1}_{t_{i+1}\leq 0 < t_i}  \int^{t_i}_0 e^{- (t_i-s)}  \sum_{\ell=0}^4 w(v_i) \chi_\ell(v_i) \int_{\mathbb{R}^3}  \chi_\ell(u) f(s,x_i-(t_i-s)v_i, u) \dd u \dd s \dd \Sigma_i   \label{bdr_K_0} \\
  & + \sum_{i=1}^{k-1} \mathbf{1}_{t_{i+1}>0} \int_{t_{i+1}}^{t_i} e^{- (t_i-s)} \sum_{\ell=0}^4 w(v_i)\chi_\ell(v_i)
  \int_{\mathbb{R}^3} \chi_\ell(u) f(s,x_i-(t_i-s)v_i, u) \dd u \dd s \dd \Sigma_i   \label{bdr_K_i} \\
  &  + \sum_{i=1}^{k-1} \mathbf{1}_{t_{i+1}\leq 0 < t_i}  \int^{t_i}_0 e^{- (t_i-s)} w(v_i) g(s,x_i-(t_i-s)v_i,v_i) \dd \Sigma_i     \label{bdr_g_0}   \\
  &  + \sum_{i=1}^{k-1} \mathbf{1}_{t_{i+1}>0} \int_{t_{i+1}}^{t_i} e^{- (t_i-s)} w(v_i) g(s,x_i-(t_i-s)v_i,v_i) \dd \Sigma_i  \bigg\}.  \label{bdr_g_i}  
\end{align}
Here $\dd \Sigma_i$ is defined as
\begin{equation}\label{Sigma}
\dd \Sigma_i = \Big\{\prod_{j=i+1}^{k-1} \dd \sigma_j \Big\}  \times  \Big\{  \frac{1}{w(v_i)\sqrt{\mu(v_i)}}   \dd \sigma_i \Big\} \times \Big\{\prod_{j=1}^{i-1} e^{-(t_j-t_{j+1})}  \dd \sigma_j \Big\} ,
\end{equation}
where $\dd \sigma_i$ is a probability measure in $\mathcal{V}_i$ \eqref{xi} given by
\begin{equation}\label{sigma_i}
\dd \sigma_i =  c_\mu\mu(v_i) (n(x_i)\cdot v_i)\dd v_i.
\end{equation}

The term \eqref{bdr_tk} corresponds to the scenario that the backward trajectory interacts with the diffuse boundary for more than $k$ times, while the other terms involve finite-time interaction. This uncontrolled interaction times is estimated by the following lemma.

\begin{lemma}\label{lemma:tk}
For $T_0>0$ sufficiently large, there exists constants $C_1,C_2>0$ independent of $T_0$ such that for $k = C_1 T_0^{5/4}$, and $(t_0,x_0,v_0) = (t,x,v)\in [0,T_0]\times \bar{\O}\times \mathbb{R}^3$,
\begin{equation*}
\int_{\prod_{j=1}^{k-1} \mathcal{V}_j}\mathbf{1}_{t_k>0}  \prod_{j=1}^{k-1} \dd \sigma_j \leq \Big( \frac{1}{2}\Big)^{C_2 T_0^{5/4}}.
\end{equation*}
\end{lemma}

\begin{proof}
Since the characteristic in Definition \ref{definition: sto cycle} with repeated interaction with the boundary is fully determined by the diffuse reflection boundary condition, this statement is independent of the equation. We refer to the proof in Lemma 23 of \cite{G} for the Boltzmann equation.

\end{proof}

To prove Proposition \ref{prop:linfty}, among the characteristic formula~\eqref{initial} - \eqref{f_bdr}, first we estimate the boundary term \eqref{f_bdr} in the following lemma.

\begin{lemma}\label{lemma:bdr}
For the boundary term~\eqref{f_bdr}, with $\lambda$ given in Proposition \ref{prop:l2}, it holds that
\begin{equation*}
\begin{split}
 w(v)|f(t_1,x_1,v)| \leq   & 4 e^{- t_1} \Vert w f_0\Vert_{L^\infty_{x,v}}+ o(1)e^{-\lambda t_1}\sup_{0\leq s\leq T_0} \Vert e^{\lambda s}w f(s)\Vert_{L^\infty_{x,v}}  \\
    &+   C(T_0) e^{-\lambda t_1}  \sup_{0\leq s\leq T_0}  \Vert e^{\lambda s} w g(s) \Vert_{L^\infty_{x,v}}  +  C(T_0) e^{-\lambda t_1} \sup_{0\leq s\leq T_0} e^{\lambda s}\Vert f(s)\Vert_{L^2_{x,v}} .
\end{split}
\end{equation*}
\end{lemma}

\begin{proof}
Since $\dd \sigma_i$ in~\eqref{sigma_i} is a probability measure, \eqref{bdr_initial} is directly bounded as
\begin{equation}\label{bdr_initial_bdd}
\eqref{bdr_initial}\leq 4 e^{- t_1}   \Vert w f_0\Vert_{L^\infty_{x,v}},
\end{equation}
{\color{black}where the constant $4$ comes from
\[\int_{\mathcal{V}_i}   |n(x_i)\cdot v_i| \sqrt{\mu(v_i)} w^{-1}(v_i) \dd v_i \leq \int_{\mathcal{V}_i}|n(x)\cdot v_i|\sqrt{\mu(v_i)}\dd v_i < 4.\]}
The exponential decay factor $e^{- t_1}$ in \eqref{bdr_initial_bdd} comes from the combinations of the decay factor in \eqref{Sigma}:
\begin{align*}
    & e^{-t_i} e^{-(t_{i-1}-t_i)} \leq e^{-t_{i-1}}, \ e^{-t_{i-1}} e^{-(t_{i-2}-t_{i-1})} \leq e^{-t_{i-2}} \cdots.
\end{align*}

For~\eqref{bdr_tk}, with $\lambda \ll 1$ and $k=C_1 T_0^{5/4}$, we apply Lemma \ref{lemma:tk} to have
\begin{equation}\label{bdr_tk_bdd}
\begin{split}
 \eqref{bdr_tk}   &  \leq \int_{\prod_{j=1}^{k-1} \mathcal{V}_j} \mathbf{1}_{t_k>0} e^{-\lambda t_k} |e^{\lambda t_k} w(v_{k-1}) f(t_k,x_k,v_{k-1})| \dd \Sigma_{k-1} \\
 &\leq o(1)e^{-\lambda t_1} \sup_{0\leq s\leq T_0}\Vert e^{\lambda s } w f(s)\Vert_{L^\infty_{x,v}}.
\end{split}
\end{equation}
Here the exponential decay factor $e^{-\lambda t_1}$ comes from the following computation:
\begin{align*}
    &     e^{-\lambda t_k} e^{-(t_{k-1}-t_k)} \leq e^{-\lambda t_{k-1}}, \  e^{-\lambda t_{k-1}} e^{-(t_{k-2}-t_{k-1})} \leq e^{-\lambda t_{k-2}} \cdots .
\end{align*}

For \eqref{bdr_g_0} and \eqref{bdr_g_i}, they are directly bounded as
\begin{equation}\label{bdr_g_bdd}
\begin{split}
|\eqref{bdr_g_0} + \eqref{bdr_g_i}| & \leq  \sup_{0\leq s\leq T_0} \Vert e^{\lambda s} wg(s)\Vert_{L^\infty_{x,v}}  \sum_{i=1}^{k-1} \int_{\prod_{j=1}^{k-1}\mathcal{V}_j} \int^{t_i}_{\max\{0,t_{i+1}\}} e^{- (t_1-s)} e^{-\lambda s} 
\dd s \dd \Sigma_i    \\
&\leq Cke^{-\lambda t_1} \sup_{0\leq s\leq T_0} e^{\lambda s}\Vert w g(s) \Vert_{L^\infty_{x,v}}   \int_0^{T_0} e^{-(T_0-s)/2}  \dd s \\ 
& \leq  Cke^{-\lambda t_1} \sup_{0\leq s\leq T_0} e^{\lambda s}\Vert w g(s) \Vert_{L^\infty_{x,v}} .
\end{split}
\end{equation}

Then we estimate \eqref{bdr_K_i}, which reads
\begin{equation}\label{iteration_i}
\begin{split}
    & \int_{\prod_{j=1}^i \mathcal{V}_j}\mathbf{1}_{t_{i+1}>0 }  \prod_{j=1}^i\dd \sigma_j \mu^{-1/2}(v_i) w^{-1}(v_i)   \\
    & \times     \int_{t_{i+1}}^{t_i} e^{-(t_1 -s)} \sum_{\ell = 0}^4 \chi_\ell (v_i) \int_{\mathbb{R}^3}  w(u)f(s,x_i-(t_i-s)v_i,u) \frac{\chi_\ell(u)}{w(u)} \dd u \dd s .
\end{split}
\end{equation}

First we decompose the $\dd s$ integral into $\mathbf{1}_{s\geq t_i-\delta} + \mathbf{1}_{s< t_i-\delta}$. The contribution of the first term reads
\begin{align}
 \eqref{iteration_i} \mathbf{1}_{s\geq t_i-\delta}   &  \leq \int_{\prod_{j=1}^{i}\mathcal{V}_j} \prod_{j=1}^{i}\dd \sigma_j \mu^{-1/2}(v_i) w^{-1}(v_i) \notag\\ 
 & \times \int^{t_i}_{\max\{t_{i+1},t_i-\delta\}} e^{-(t_1-s)} \sum_{\ell = 0}^4  w(v_i) \chi_\ell (v_i) \int_{\mathbb{R}^3}  w(u) f(s,x_i-(t_i-s)v_i,u) \frac{\chi_\ell(u)}{w(u)} \dd u \dd s  \notag\\
 &\leq o(1)e^{-\lambda t_1} \sup_{0\leq s\leq T_0}\Vert e^{\lambda s} w f(s)\Vert_{L^\infty_{x,v}} . \label{bdr_s_small_bdd}
\end{align}

Then we decompose the $\dd u$ integral into $\mathbf{1}_{|u|\geq N }  + \mathbf{1}_{|u|<N}$. The contribution of the first term reads
\begin{align}
    & \eqref{iteration_i} \mathbf{1}_{|u|\geq N}  \leq \int_{\prod_{j=1}^{i}\mathcal{V}_j} \prod_{j=1}^{i}\dd \sigma_j  \mu^{-1/2}(v_i) w^{-1}(v_i)  \notag\\
    &\ \ \times  \int_{t_{i+1}}^{t_i} e^{-(t_1 -s)} 
 \sum_{\ell = 0}^4 \chi_\ell(u) \int_{\mathbb{R}^3}  \mathbf{1}_{|u|\geq N}  w(u) f(s,X_i(s),u) \frac{\chi_\ell(u)}{w(u)} \dd s \notag\\
    & \leq o(1) e^{-\lambda t_1} \sup_{0\leq s \leq T_0}\Vert e^{\lambda s}w f(s)\Vert_{L^\infty_{x,v}}.   \label{bdr_u_large_bdd}
\end{align}

Last, we consider the intersection of all other cases, where we have $s<t_i-\delta$, and $|u|<N$. We compute
\begin{align}
  & \eqref{iteration_i} \mathbf{1}_{s<t_i-\delta} \mathbf{1}_{|u|<N}  \notag\\
  &\leq   \int_{\prod_{j=1}^{i-1}\mathcal{V}_j} \prod_{j=1}^{i-1}\dd \sigma_j \int_{\mathcal{V}_i}  \sqrt{\mu(v_i)} |n(x_i)\cdot v_i| \dd v_i   \notag\\ 
  &\times \int_{t_{i+1}}^{t_i-\delta}  e^{-(t_i -s)} \int_{|u|\leq N}  f(s,x_i-(t_i-s)v_i,u) \dd u\dd s  \notag \\
  & \leq \frac{1}{\delta^3}\int_{\prod_{j=1}^{i-1}\mathcal{V}_j} \prod_{j=1}^{i-1}\dd \sigma_j  \int^{t_i-\delta}_{0}  e^{-(t_1-s)}     \int_{|u|\leq N} \int_{\O} f(s,y,u) \dd u \dd y \dd s.       \label{other_case}
\end{align}
In the last line, we have applied the change of variable $v_i \to y = x_i-(t_i-s)v_i \in \O$ with Jacobian 
\[\Big|\det\Big(\frac{\p x_i-(t_i-s)v_i}{\p v_i} \Big) \Big| = (t_i-s)^3 \geq \delta^3 .\]

Then we leverage the $L^2$ estimate by applying the H\"older inequality
\begin{align}
  \eqref{other_case}    & \leq  C_{N,\delta,\O}   \int_{\prod_{j=1}^{i-1}\mathcal{V}_j} \prod_{j=1}^{i-1}\dd \sigma_j \times    \int_{0}^{t_1} e^{-(t_1-s)} \Vert f(s)\Vert_{L^2_{x,v}} \dd s \notag\\
  & \leq  C_{N,\delta,\O} e^{-\lambda t_1} \sup_{0\leq s\leq T_0} e^{\lambda s} \Vert f(s)\Vert_{L^2_{x,v}}  . \label{other_case_bdd}
\end{align}

Collecting \eqref{bdr_s_small_bdd}, \eqref{bdr_u_large_bdd} and \eqref{other_case_bdd}, we conclude that
\begin{equation}\label{bdr_K_i_bdd}
\eqref{bdr_K_i} \lesssim o(1) e^{-\lambda t_1} \sup_{0\leq s\leq T_0} \Vert e^{\lambda s}w f(s)\Vert_{L^\infty_{x,v}} + C_{N,\delta,k,\O} e^{-\lambda t_1} \sup_{0\leq s\leq T_0} e^{\lambda s} \Vert f(s)\Vert_{L^2_{x,v}}. 
\end{equation}

By the same computation, we obtain the same bound for \eqref{bdr_K_0}:
\begin{equation}\label{bdr_K_0_bdd}
\eqref{bdr_K_0} \lesssim o(1) e^{-\lambda t_1} \sup_{0\leq s\leq T_0} \Vert e^{\lambda s}w f(s)\Vert_{L^\infty_{x,v}} + C_{N,\delta,k,\O} e^{-\lambda t_1} \sup_{0\leq s\leq T_0} e^{\lambda s} \Vert f(s)\Vert_{L^2_{x,v}}.  
\end{equation}

Summarizing \eqref{bdr_initial_bdd}, \eqref{bdr_tk_bdd}, \eqref{bdr_g_bdd}, \eqref{bdr_K_0_bdd} and \eqref{bdr_K_i_bdd}, and using the fact that $k$ is a function of $T_0$, we conclude the lemma with $C(T_0) = C_{N,\delta,\O,k}$.
\end{proof}

Now we are in a position to prove Proposition \ref{prop:linfty}.

\begin{proof}[\textbf{Proof of Proposition \ref{prop:linfty}}]

We focus on the a priori estimate. We will discuss the construction of solution using an approximating sequence at the end of the proof.

First of all, \eqref{initial}, \eqref{g_0} and \eqref{g_1} are bounded as
\begin{equation}\label{initial_g_bdd}
|\eqref{initial}| + |\eqref{g_0}| + |\eqref{g_1}| \leq e^{- T_0} \Vert w f_0\Vert_{L^\infty_{x,v}} + Ce^{- \lambda T_0}\sup_{0\leq s\leq T_0} \Vert e^{\lambda s} w g(s) \Vert_{L^\infty_{x,v}}.
\end{equation}
Moreover, \eqref{f_bdr} is bounded by Lemma \ref{lemma:bdr} as
\begin{equation}\label{f_bdr_bdd}
\begin{split}
 \eqref{f_bdr} \leq   &  4 e^{- T_0} \Vert w f_0\Vert_{L^\infty_{x,v}} + o(1)e^{-\lambda T_0}\sup_{0\leq s\leq T_0}\Vert e^{\lambda s} w f(s)\Vert_{L^\infty_{x,v}}  \\
    &  + C(T_0) \big[e^{-\lambda T_0} \sup_{0\leq s\leq T_0}\Vert e^{\lambda s} w g(s) \Vert_{L^\infty_{x,v}} +  e^{-\lambda T_0}\sup_{0\leq s\leq T_0} \Vert f(s)\Vert_{L^2_{x,v}} \big].
\end{split}
\end{equation}

Then we focus on~\eqref{K_1}. We expand $f(s,x-(t-s)v,u)$ using \eqref{initial} - \eqref{f_bdr} again along the characteristic with velocity $u$.

Denote $t^u_1:= s - \tb(x-(t-s)v,u)$ and $y:=x-(t-s)v$, we compute that
\begin{align}
    & \eqref{K_1} = \mathbf{1}_{t_1>0} \int^{T_0}_{t_1} \dd s e^{-(T_0 -s)}  \sum_{i=0}^4 w(v)\chi_i(v)\int_{\mathbb{R}^3} \dd u  \frac{\chi_i(u)}{w(u)} \notag \\
    &\times \Big\{ \mathbf{1}_{t_1^u \leq 0} e^{- s} w(u)f(0,y-su,u)       \label{u_initial} \\
    &  +\mathbf{1}_{t_1^u\leq 0} \int_0^{s}  e^{-(s-s')} \dd s'   \sum_{j=0}^4 w(u)\chi_j(u)  \int_{\mathbb{R}^3} w(u') f(s',y-(s-s')u, u') \frac{\chi_j(u')}{w(u')}    \dd u' \label{u_K0} \\
    &  +\mathbf{1}_{t_1^u> 0} \int_{t_1^u}^{s}  e^{-(s-s')} \dd s'  \sum_{j=0}^4 w(u)\chi_j(u)  \int_{\mathbb{R}^3} w(u')  f(s',y-(s-s')u, u') \frac{\chi_j(u')}{w(u')}  \dd u' \label{u_K1}\\
    & + \mathbf{1}_{t_1^u\leq 0} \int_{0}^{s} e^{-(s-s')} w(u) g(s',y-(s-s')u,u) \dd s' \label{u_g0} \\
    & + \mathbf{1}_{t_1^u> 0} \int_{t_1^u}^{s} e^{-(s-s')} w(u) g(s',y-(s-s')u,u) \dd s' \label{u_g1} \\
    & + \mathbf{1}_{t_1^u > 0} e^{-(s-t_1^u)} w(u) f(t_1^u,y-\tb(y,u)u,u) \Big\}.   \label{u_bdr}
\end{align}

The contribution of \eqref{u_initial} in \eqref{K_1} is bounded by
\begin{align}
    & \int^{T_0}_{t_1} \dd s e^{- (T_0-s)} 
   \sum_{i=0}^4 w(v)\chi_i(v)  \int_{\mathbb{R}^3} \dd u \frac{\chi_i(u)}{w(u)}  e^{- s} \Vert w f_0\Vert_{L^\infty_{x,v}}    \notag \\
    &\leq C_\beta\int_{t_1}^{T_0} \dd s e^{-(T_0-s)/2} e^{-\frac{1}{2}T_0} \Vert wf_0\Vert_{L^\infty_{x,v}}  \leq C_\beta e^{-\frac{ T_0}{2}} \Vert w f_0\Vert_{L^\infty_{x,v}} . \label{u_initial_bdd}
\end{align}
Here the constant $C_\beta$ comes from $(1+|v|)^{\beta}e^{\theta |v|^2}\chi_i(v)\lesssim C_\beta$.

The contribution of \eqref{u_g0} and \eqref{u_g1} in \eqref{K_1} are bounded by
\begin{align}
    & \int_{t_1}^{T_0} \dd s e^{-(T_0-s)} \sum_{i=0}^4 w(v)\chi_i(v) \int_{\mathbb{R}^3} \dd u  \frac{\chi_i(u)}{w(u)} \int^s_0 \dd s' e^{-(s-s')}  e^{-\lambda s'} \sup_{0\leq s\leq T_0} \Vert  e^{\lambda s}wg(s)\Vert_{L^\infty_{x,v}}  \notag  \\
    & \leq C \sup_{0\leq s\leq T_0}\Vert e^{\lambda s} wg(s)\Vert_{L^\infty_{x,v}} \int_{t_1}^{T_0} \dd s e^{-(T_0-s)} e^{-\lambda s}\int_{\mathbb{R}^3} \dd u  \frac{\chi_i(u)}{w(u)} \notag \\
    &\leq C \sup_{0\leq s\leq T_0}\Vert  e^{\lambda s}wg(s)\Vert_{L^\infty_{x,v}}. \label{u_g_bdd}
\end{align}

The contribution of the boundary term in~\eqref{u_bdr} can be bounded by applying Lemma \ref{lemma:bdr}:
\begin{align}
 | \eqref{u_bdr} | & \leq \int_{t_1}^{T_0} \dd s e^{-(T_0-s)} \sum_{i=0}^4 w(v)\chi_i(v) \int_{\mathbb{R}^3} \dd u \frac{\chi_i(u)}{w(u)} e^{- (s-t_1^u)} \notag \\
 & \times \Big\{ e^{- t_1^u} 4\Vert w f_0\Vert_{L^\infty_{x,v}} + o(1)e^{-\lambda t_1^u}\sup_{0\leq s\leq T_0}\Vert e^{\lambda s} w f(s)\Vert_{L^\infty_{x,v}} \notag \\
  &+ C(T_0) \big[e^{-\lambda t_1^u}\sup_{0\leq s\leq T_0} \Vert e^{\lambda s}w g(s) \Vert_{L^\infty_{x,v}}  +  e^{-\lambda t_1^u} \sup_{0\leq s \leq T_0} e^{\lambda s}\Vert f(s)\Vert_{L^2_{x,v}}  \big] \Big\} \notag\\
  &\leq  C_\beta \int_{t_1}^{T_0} \dd s e^{-(T_0-s)}    \times \Big[ 4e^{-\frac{ s}{2}} \Vert w f_0\Vert_{L^\infty_{x,v}} + o(1)e^{-\lambda s}\sup_{0\leq s\leq T_0} \Vert e^{\lambda s} wf(s)\Vert_{L^\infty_{x,v}}  \notag\\
  & + C(T_0) \big[e^{-\lambda s}\sup_{0\leq s\leq T_0} \Vert e^{\lambda s} wg(s) \Vert_{L^\infty_{x,v}}  +  e^{-\lambda s} \sup_{0\leq s \leq T_0} e^{\lambda s}\Vert f(s)\Vert_{L^2_{x,v}}  \big]   \Big]     \notag\\
  &\leq   4C_\beta e^{-\frac{ T_0}{2}} \Vert w f_0\Vert_{L^\infty_{x,v}} + o(1) e^{-\lambda T_0}\sup_{0\leq s\leq T_0}\Vert e^{\lambda s}w f(s)\Vert_{L^\infty_{x,v}}  \notag\\
  & + C(T_0)\big[e^{-\lambda T_0}\sup_{0\leq s\leq T_0} \Vert e^{\lambda s}wg(s) \Vert_{L^\infty_{x,v}}  +   e^{-\lambda T_0} \sup_{0\leq s \leq T_0} e^{\lambda s}\Vert f(s)\Vert_{L^2_{x,v}} \big].   \label{u_bdr_bdd}
\end{align}
Again, the constant $C_\beta$ comes from $w(v)\chi_i(v)\lesssim C_\beta$.

Then we focus on the contribution of \eqref{u_K1} in \eqref{K_1}. First we decompose the $\dd s' $ integral into $\mathbf{1}_{s-s'< \delta} + \mathbf{1}_{s-s'\geq \delta}$. The contribution of the first term reads
\begin{align}
    & |\eqref{u_K1} \mathbf{1}_{s-s'<\delta} |\notag \\
   & \leq   \int_{t_1}^{T_0} \dd s e^{-(T_0-s)} \sum_{i=0}^4 w(v)\chi_i(v)    \int_{\mathbb{R}^3} \dd u   \frac{\chi_i(u)}{w(u)}   \int^s_{\max\{s-\delta,t_1^u\}} e^{-(s-s')} \dd s' \notag \\
   & \times \sum_{j=0}^4 w(u)\chi_j(u) \int_{\mathbb{R}^3} \dd u' \frac{\chi_j(u')}{w(u')}         e^{-\lambda s'} \sup_{0\leq s\leq T_0}\Vert e^{\lambda s} w f(s)\Vert_{L^\infty_{x,v}}   \notag \\
   & \leq o(1)\int_{t_1}^{T_0} \dd s e^{-(T_0-s)}  \sum_{i=0}^4 w(v)\chi_i(v) \int_{\mathbb{R}^3} \dd u \frac{\chi_i(u)}{w(u)} e^{-\lambda s} \sup_{0\leq s\leq T_0}\Vert e^{\lambda s} w f(s)\Vert_{L^\infty_{x,v}} \notag \\
    & \leq    o(1)e^{-\lambda T_0}\sup_{0\leq s\leq T_0}\Vert e^{\lambda s} w f(s)\Vert_{L^\infty_{x,v}}. \label{u_K1_s_small}
\end{align}

Next we decompose the $\dd u'$ integral into $\mathbf{1}_{|u'|\geq N} + \mathbf{1}_{|u'|\leq N}$. The contribution of the first term reads
\begin{align}
   & |\eqref{u_K1}\mathbf{1}_{|u'|\geq N } |\notag\\
   & \leq o(1)\int_{t_1}^{T_0} \dd s e^{-(T_0-s)} \sum_{i=0}^4 w(v)\chi_i(v) \int_{\mathbb{R}^3} \dd u \frac{\chi_i(u)}{w(u)} \int^s_{t_1^u} e^{-(s-s')}  e^{-\lambda s'}\dd s' \sup_{0\leq s\leq T_0}\Vert e^{\lambda s} w f(s)\Vert_{L^\infty_{x,v}}   \notag\\
   & \leq   o(1)\int_{t_1}^{T_0} \dd s e^{-(T_0-s)}  e^{-\lambda s} \sup_{0\leq s\leq T_0} \Vert e^{\lambda s} w f(s)\Vert_{L^\infty_{x,v}}    \notag\\
   & \leq  o(1) e^{-\lambda T_0} \sup_{0\leq s\leq T_0}\Vert e^{\lambda s}w f(s)\Vert_{L^\infty_{x,v}}. \label{u_K1_u_prime_small}
\end{align}

Now we consider the intersection of all other cases, where we have $s'<s-\delta$ and $|u'|<N$. In such case, we compute such contribution in \eqref{u_K1} as
\begin{align}
    &  |\eqref{u_K1}\mathbf{1}_{s'<s-\delta, \ |u'|<N} |\notag\\
    &  \leq  \int_{t_1}^{T_0} \dd s e^{-(T_0 -s)} \int_{\mathbb{R}^3} \dd u \frac{\chi_i(u)}{w(u)} \int^{s-\delta}_{t_1^u} e^{-(s-s')} \dd s' \int_{|u'|<N} \dd u'   f(s', y-(s-s')u,u') . \label{u_K1_other}
\end{align}
We apply the change of variable $u\to y'=y-(s-s')u$ with Jacobian 
\[\Big|\det\Big(\frac{\p [y-(s-s')u]}{\p u} \Big) \Big| = (s-s')^3 \geq \delta^3 \]
to derive that
\begin{align}
 & |\eqref{u_K1_other}| \leq      \int_{t_1}^{T_0} \dd s e^{-(T_0-s)} \int_{\O} \dd y'  \int_{t_1^u}^{s-\delta} e^{-(s-s')} \dd s' \int_{|u'|<N} f(s',y',u') \dd u'           \notag \\
  & \leq C_{N,\O} \int_{t_1}^{T_0} \dd s e^{-(T_0-s)}  \int_{0}^{s-\delta} e^{-(s-s')} \Vert f(s')\Vert_{L^2_{x,v}} \dd s'     \notag \\
  & \leq C_{N,\O}  \sup_{0\leq s\leq T_0}  e^{\lambda s}\Vert f(s)\Vert_{L^2_{x,v}} \int_{t_1}^{T_0} \dd s e^{-(T_0-s)}   e^{-\lambda s} \notag \\
  &\leq  C_{N,\O} e^{-\lambda T_0} \sup_{0\leq s\leq T_0} e^{\lambda s}\Vert f(s)\Vert_{L^2_{x,v}} .  \label{u_K1_other_bdd}
\end{align}

Collecting \eqref{u_K1_s_small}, \eqref{u_K1_u_prime_small} and \eqref{u_K1_other_bdd}, we conclude
\begin{equation}\label{u_K1_bdd}
|\eqref{u_K1}| \leq o(1) e^{-\lambda T_0} \sup_{0\leq s\leq T_0} \Vert e^{\lambda s}w f(s)\Vert_{L^\infty_{x,v}} + C(T_0)  e^{-\lambda T_0}\sup_{0\leq s\leq T_0} e^{\lambda s}\Vert f(s)\Vert_{L^2_{x,v}} .
\end{equation}

By the same computation, we obtain that
\begin{equation}\label{u_K0_bdd}
 |\eqref{u_K0}| \leq  o(1) e^{-\lambda T_0} \sup_{0\leq s\leq T_0} \Vert e^{\lambda s}w f(s)\Vert_{L^\infty_{x,v}} + C(T_0) e^{-\lambda T_0}\sup_{0\leq s\leq T_0} e^{\lambda s}\Vert f(s)\Vert_{L^2_{x,v}}.   
\end{equation}

We combine \eqref{u_initial_bdd}, \eqref{u_g_bdd}, \eqref{u_bdr_bdd}, \eqref{u_K0_bdd} and \eqref{u_K1_bdd} to conclude the estimate for \eqref{K_1}:
\begin{equation}\label{K1_bdd}
\begin{split}
   \eqref{K_1} & \leq (4+5C_\beta)e^{-\frac{ T_0}{2}} \Vert w f_0\Vert_{L^\infty_{x,v}} + o(1)e^{-\lambda T_0}\sup_{0\leq s\leq T_0} \Vert e^{\lambda s}w f(s)\Vert_{L^\infty_{x,v}} \notag \\
   &+ C(T_0) e^{-\lambda T_0}\sup_{0\leq s\leq T_0} \Vert e^{\lambda s}w g(s) \Vert_{L^\infty_{x,v}}   +C(T_0) e^{-\lambda T_0}\sup_{0\leq s\leq T_0} e^{\lambda s}\Vert f(s)\Vert_{L^2_{x,v}}.
\end{split}
\end{equation}

Similarly, we can have the same estimate for \eqref{K_0} as
\begin{equation}\label{K0_bdd}
\begin{split}
   |\eqref{K_0}| & \leq (4+5C_\beta)e^{-\frac{ T_0}{2}} \Vert w f_0\Vert_{L^\infty_{x,v}} + o(1)e^{-\lambda T_0}\sup_{0\leq s\leq T_0} \Vert e^{\lambda s}w f(s)\Vert_{L^\infty_{x,v}}    \\
    & + C(T_0) e^{-\lambda T_0}\sup_{0\leq s\leq T_0} \Vert e^{\lambda s} w g(s) \Vert_{L^\infty_{x,v}} +C(T_0)  e^{-\lambda T_0}\sup_{0\leq s\leq T_0} e^{\lambda s}\Vert f(s)\Vert_{L^2_{x,v}}.
\end{split}
\end{equation}

Last we collect \eqref{initial_g_bdd}, \eqref{f_bdr_bdd}, \eqref{K1_bdd} and \eqref{K0_bdd} to conclude that
\begin{align}
 &w(v) |f(T_0,x,v)|   \notag\\
 &   \leq (5+5C_\beta)e^{-\frac{ T_0}{2}} \Vert w f_0\Vert_{L^\infty_{x,v}} +o(1)e^{-\lambda T_0}\sup_{0\leq s\leq T_0} \Vert e^{\lambda s}w f(s)\Vert_{L^\infty_{x,v}}     \label{C_theta} \\
  &  + C(T_0) e^{-\lambda T_0}\sup_{0\leq s\leq T_0}\Vert e^{\lambda s}w g(s) \Vert_{L^\infty_{x,v}} +C(T_0) e^{-\lambda T_0}\sup_{0\leq s\leq T_0} e^{\lambda s}\Vert f(s)\Vert_{L^2_{x,v}}.  \notag
\end{align}

Since the source term $g$ and initial condition $f_0$ satisfy \eqref{linfty_condition}, the conditions in Proposition \ref{prop:l2} are satisfied. {\color{black}With the weight $w(v)=(1+|v|)^\beta e^{\theta |v|^2}$, we control the $L^2$ term by Proposition \ref{prop:l2}:
\begin{align}
   \sup_{0\leq s\leq T_0}e^{\lambda s}\Vert f(s)\Vert_{L^2_{x,v}} & \lesssim \Vert f_0\Vert_{L^2_{x,v}} + \Big(\int_0^{T_0} e^{-2\lambda s}\Vert e^{2\lambda s}g(s)\Vert_{L^2_{x,v}}^2\dd s   \Big)^{1/2}\notag \\
   & \lesssim [\Vert wf_0\Vert_{L^\infty_{x,v}} + \sup_{0\leq s\leq T_0}\Vert e^{2\lambda s}wg(s)\Vert_{L^\infty_{x,v}}] \Vert w^{-1}(v)\Vert_{L^2_{x,v}} \notag\\
   &
   \lesssim  \Vert wf_0\Vert_{L^\infty_{x,v}} +  \sup_{0\leq s\leq T_0} \Vert e^{2\lambda s} w g(s)\Vert_{L^\infty_{x,v}} . \label{l2_term}
\end{align}
Here we have applied the definition of $w(v)$ in \eqref{weight} such that $w^{-1}(v)\in L^2_v$.}

For given $0\leq t<\infty$, we denote
\begin{align*}
    \mathcal{R}_t := \Vert wf_0\Vert_{L^\infty_{x,v}} + \sup_{0\leq s\leq t}\Vert e^{2\lambda s} w g(s) \Vert_{L^\infty_{x,v}}.
\end{align*}

Recall that $C_\beta$ in \eqref{C_theta} does not depend on $T_0$. We choose $T_0$ to be large enough such that  $(5+5C_\beta)e^{-\frac{ T_0}{2}} < e^{-\frac{ T_0}{4}}$. Then we further have
\begin{align}
   \Vert  wf(T_0)\Vert_{L^\infty_{x,v}}  &  \leq  e^{-\frac{ T_0}{4}} \Vert wf_0\Vert_{L^\infty_{x,v}}    + o(1)e^{-\lambda T_0}\sup_{0\leq s\leq T_0} \Vert e^{\lambda s}w f(s)\Vert_{L^\infty_{x,v}} \notag \\
   &+  C(T_0) e^{-\lambda T_0} \sup_{0\leq s\leq T_0} \Vert e^{2\lambda s} w g(s) \Vert_{L^\infty_{x,v}} + C(T_0)e^{-\lambda T_0}\sup_{0\leq s\leq T_0} e^{\lambda s}\Vert f(s)\Vert_{L^2_{x,v}}.  \label{est_T0}
\end{align}

For $0\leq t\leq T_0$, with the same choice of $k=C_1 T_0^{5/4}$, it is straightforward to apply the same argument for $e^{\lambda t}w(v)|f(t,x,v)|$ to have:
\begin{align}
  \Vert wf(t)\Vert_{L^\infty_{x,v}}  & \leq (5+5C_\beta)e^{-\frac{ t}{2}}\Vert wf_0\Vert_{L^\infty_{x,v}} + o(1) e^{-\lambda t} \sup_{0\leq s\leq t}\Vert e^{\lambda s} wf(s)\Vert_{L^\infty_{x,v}} \notag\\
  & + C(T_0)e^{-\lambda t} \sup_{0\leq s\leq t} \Vert e^{2\lambda s} w g(s) \Vert_{L^\infty_{x,v}} + C(T_0)e^{-\lambda t}\sup_{0\leq s\leq t} e^{\lambda s}\Vert f(s)\Vert_{L^2_{x,v}}. \label{est_t}
\end{align} 

For $t=mT_0$, we apply \eqref{est_T0} to have
\begin{align}
  &\Vert wf(mT_0)\Vert_{L^\infty_{x,v}}  \notag\\
  &  \leq e^{-\frac{ T_0}{4}} \Vert wf((m-1)T_0)\Vert_{L^\infty_{x,v}} + C(T_0)e^{-\lambda T_0} \sup_{0 \leq s\leq T_0} \Vert e^{\lambda s} g((m-1)T_0+s) \Vert_{L^\infty_{x,v}}   \notag\\
  & + o(1)e^{-\lambda T_0} \sup_{0 \leq s\leq T_0} \Vert e^{\lambda s }w f((m-1)T_0+s) \Vert_{L^\infty_{x,v}} + C(T_0) e^{-\lambda T_0} \sup_{0\leq s\leq T_0}e^{\lambda s}\Vert f((m-1)T_0+s)\Vert_{L^2_{x,v}}  \notag\\
  & \leq e^{-\frac{ T_0}{4}} \Vert wf((m-1)T_0)\Vert_{L^\infty_{x,v}} +o(1)e^{-\lambda m T_0} \sup_{0 \leq s\leq mT_0} \Vert e^{\lambda s }w f(s) \Vert_{L^\infty_{x,v}} + C(T_0) e^{-\lambda m T_0} \mathcal{R}_{mT_0} \notag\\
  & \leq e^{-2\frac{ T_0}{4}} \Vert wf((m-2)T_0)\Vert_{L^\infty_{x,v}} \notag \\
  &+  e^{-\lambda m T_0} \Big[o(1)\sup_{0\leq s\leq mT_0} \Vert e^{\lambda s}w f(s)\Vert_{L^\infty_{x,v}} + C(T_0) \mathcal{R}_{mT_0} \Big]\times \big[1 + e^{-\frac{(1-4\lambda) T_0}{4}} \big] \notag\\
  & \leq \cdots \leq e^{-\frac{mT_0}{4}} \Vert wf_0\Vert_{L^\infty_{x,v}} \notag\\
  & + e^{-\lambda m T_0} \Big[o(1)\sup_{0\leq s\leq mT_0} \Vert e^{\lambda s}w f(s)\Vert_{L^\infty_{x,v}} + C(T_0) \mathcal{R}_{mT_0} \Big]\times \sum_{i=0}^{m-1}  e^{-\frac{i(1-4\lambda)T_0}{4}} \notag\\
  & \leq  o(1)e^{-\lambda m T_0} \sup_{0\leq s\leq mT_0}\Vert e^{\lambda s}wf(s)\Vert_{L^\infty_{x,v}} + C(T_0)e^{-\lambda m T_0} \mathcal{R}_{mT_0}. \label{mT0}
\end{align}
In the fourth line, we have applied the same computation as \eqref{l2_term} to the $L^2$ term.

For any $t>0$, we can choose $m$ such that $mT_0\leq t\leq (m+1)T_0$. With $t=mT_0+s$, $0\leq s\leq T_0$, we apply \eqref{est_t} to have
\begin{align}
  & \Vert wf(t) \Vert_{L^\infty_{x,v}}  = \Vert wf(mT_0 + s)\Vert_{L^\infty_{x,v}} \notag\\
  & \leq (5+5C_\beta)e^{\frac{- s}{2}}\Vert wf(mT_0)\Vert_{L^\infty_{x,v}}  + o(1)e^{-\lambda s} \sup_{0\leq s'\leq s} \Vert e^{\lambda s'}wf(mT_0+s')\Vert_{L^\infty_{x,v}} \notag\\
  &+ C(T_0) e^{-\lambda s}\sup_{0\leq s'\leq s}   \Vert e^{2\lambda s'}w g(mT_0+s') \Vert_{L^\infty_{x,v}} + C(T_0) e^{-\lambda s} \sup_{0\leq s'\leq s}e^{\lambda s'}\Vert f(mT_0+s')\Vert_{L^2_{x,v}} \notag\\
  & \leq o(1)(5+5C_\beta)e^{-\lambda (m T_0+s)} \sup_{0\leq s\leq mT_0}\Vert e^{\lambda s}wf(s)\Vert_{L^\infty_{x,v}} + C(T_0)e^{-\lambda (m T_0+s)} \mathcal{R}_{mT_0+s} \notag\\
  & \leq o(1)e^{-\lambda t} \sup_{0\leq s\leq t}\Vert e^{\lambda s}wf(s)\Vert_{L^\infty_{x,v}} + C(T_0) e^{-\lambda t} \mathcal{R}_t. \label{f_t_bdd}
\end{align}
In the fourth line, we have applied \eqref{mT0} and \eqref{l2_term} to the $L^2$ term.

Since \eqref{f_t_bdd} holds for all $t$, we conclude that
\begin{align}
  e^{\lambda t}\Vert wf(t)\Vert_{L^\infty_{x,v}}  & \leq C(T_0) e^{-\lambda t}\big[\Vert wf_0\Vert_{L^\infty_{x,v}} + \sup_{0\leq s\leq t}\Vert e^{2\lambda s}w g(s)\Vert_{L^\infty_{x,v}}\big]. \label{apriori}
\end{align}

We conclude the a-priori estimate. To establish the existence of the solution, we will use the following approximating sequence:
\begin{align}
\begin{cases}
        &  \p_t f^{\ell+1} + v\cdot \nabla_x f^{\ell+1} + f^{\ell+1} = \mathbf{P}f^\ell + g, \ f^{\ell+1}(0,x,v) = f_0(x,v) \\
    &  f^{\ell+1}|_{\gamma_-} = (1-\frac{1}{j})c_\mu\sqrt{\mu(v)} \int_{n(x)\cdot u>0} f^{\ell}(u)\sqrt{\mu(u)}(n(x)\cdot u) \dd u.
\end{cases} \label{cauchy_linear}
\end{align}
By employing a similar argument using the method of characteristic, one can show that $f^{\ell}$ forms a Cauchy sequence in the $L^\infty$ space. This leads to the existence of a solution $f$ that satisfies \eqref{apriori}. The uniqueness follows in a similar way. For conciseness, we do not present the detail of such computation, we refer to a detailed argument in Proposition 7.1 of \cite{EGKM}.

We conclude the proof of Proposition \ref{prop:linfty}. 
\end{proof}

\subsection{\textbf{Proof of Theorem \ref{thm:linfty}}}\label{sec:thm_proof}

We consider the following iteration sequence:
\begin{align*}
\begin{cases}
& \p_t f^{\ell+1} + v\cdot \nabla_x f^{\ell+1} + \mathcal{L} f^{\ell+1} = \Gamma(f^\ell), \ f^{\ell+1}(0,x,v) = f_0(x,v), \\
&  f^{\ell+1}|_{\gamma_-} = c_\mu\sqrt{\mu(v)}\int_{n(x)\cdot u>0} f^{\ell+1}\sqrt{\mu(u)}(n(x)\cdot u) \dd u.    
\end{cases}
\end{align*}
The initial sequence is defined as $f^0 = 0$. With the assumption on the initial condition $\Vert wf_0\Vert_{L^\infty_{x,v}}<\delta$, we apply Proposition \ref{prop:linfty} to conclude that for $\ell=0$, there exists a unique solution $f^1$ such that
\begin{align*}
    &    \sup_{0\leq s\leq t}\Vert e^{\lambda s}wf^{1}\Vert_{L^\infty_{x,v}} \leq \delta.
\end{align*}
Inductively, we assume $\sup_{0\leq s\leq t}\Vert e^{\lambda s}w f^\ell\Vert_{L^\infty_{x,v}} \leq 2C\delta$. Then the condition in Lemma \ref{lemma:gamma_property} is satisfied. Moreover, from Lemma \ref{lemma:Pgamma}, we have $\mathbf{P}(\Gamma(f^\ell))=0$, thus the condition of Proposition \ref{prop:linfty} is also satisfied.

We apply Proposition \ref{prop:linfty} to conclude that there is a unique solution $f^{\ell+1}$ such that
\begin{align*}
  \sup_{0\leq s\leq t} \Vert e^{\lambda s} wf^{\ell+1}\Vert_{L^\infty_{x,v}}  & \leq C\Vert wf_0\Vert_{L^\infty_{x,v}} + C\Big[ \sup_{0\leq s\leq t} \Vert e^{\lambda s}w f^{\ell}\Vert_{L^\infty_{x,v}}^2 + \sup_{0\leq s\leq t} \Vert e^{\lambda s}w f^{\ell}\Vert_{L^\infty_{x,v}}^3 \Big]  .
\end{align*}
Here we have used Lemma \ref{lemma:gamma_property} for $1\leq i\leq 3$ to obtain the following estimates:
\begin{align*}
 \sup_{0\leq s\leq t}\Vert  e^{2\lambda s} w \Gamma_i(f^\ell) \Vert_{L^\infty_{x,v}} \lesssim  \sup_{0\leq s\leq t} \Vert e^{\lambda s}w f^{\ell} \Vert_{L^\infty_{x,v}}^2 ,
\end{align*}
and
\begin{align*}
    &   \sup_{0\leq s\leq t}\Vert e^{2\lambda s}w\Gamma_4(f^\ell) \Vert_{L^\infty_{x,v}}  \lesssim \sup_{0\leq s\leq t} \Vert e^{\lambda s} w f^\ell \Vert_{L^\infty_{x,v}}^3.
\end{align*}

We take $\Vert wf_0\Vert_{L^\infty_{x,v}}<\delta$ to be small enough such that $2C\delta \ll 1$, then with $\Vert e^{\lambda s}w f^\ell\Vert_{L^\infty_{x,v}} \leq 2C\delta$, we further derive that
\begin{align*}
    \sup_{0\leq s\leq t} \Vert e^{\lambda s} wf^{\ell + 1} \Vert_{L^\infty_{x,v}} \leq C\delta + 4C^2\delta^2 + 8C^3 \delta^3 \leq 2C\delta.
\end{align*}
Hence by induction argument, we conclude the uniform-in-$\ell$ estimate:
\begin{equation}\label{uniform-in-ell}
\sup_{\ell} \sup_{0\leq s\leq t} \Vert e^{\lambda s} w f^{\ell} \Vert_{L^\infty_{x,v}} \leq 2C\delta.
\end{equation}

Next, we take the difference $f^{\ell+1}-f^\ell$. The equation of $f^{\ell+1}-f^\ell$ becomes
\begin{align*}
\begin{cases}
    & \p_t (f^{\ell+1} - f^\ell) + v\cdot \nabla_x (f^{\ell+1}-f^\ell) + \mathcal{L}(f^{\ell+1}-f^\ell) = \Gamma(f^\ell) - \Gamma(f^{\ell-1}), \\
    & f^{\ell+1}(0,x,v) - f^\ell(0,x,v) = 0, \\
    & [f^{\ell+1} - f^\ell]|_{\gamma_-} = c_\mu\sqrt{\mu(v)}\int_{n(x)\cdot u>0} [f^{\ell+1} - f^{\ell}] \sqrt{\mu(u)}(n(x)\cdot u)\dd u.
\end{cases}
\end{align*}

We apply Proposition \ref{prop:linfty} to have 
\begin{align*}
  &\sup_{0\leq s\leq t}\Vert e^{\lambda s} w (f^{\ell+1}-f^\ell)\Vert_{L^\infty_{x,v}}    \leq C \sup_{0\leq s\leq t}\Vert e^{2\lambda s} w[\Gamma(f^\ell) - \Gamma(f^{\ell-1}) ]\Vert_{L^\infty_{x,v}} \\
  & \lesssim  \delta \sup_{0\leq s\leq t} \Vert e^{\lambda s}w(f^\ell-f^{\ell-1})\Vert_{L^\infty_{x,v}} .
\end{align*}
In the second line, we have applied the estimate in Lemma \ref{lemma:nonlinear_substraction} to $\Gamma(f^\ell)-\Gamma(f^{\ell-1})$. The condition in Lemma \ref{lemma:nonlinear_substraction} is satisfied due to the uniform-in-$\ell$ estimate \eqref{uniform-in-ell}.

Thus for some constant $C_1$, we have
\begin{align*}
    &   \sup_{0\leq s\leq t}\Vert e^{\lambda s} w (f^{\ell+1}-f^\ell)\Vert_{L^\infty_{x,v}}  \leq C_1 \delta \sup_{0\leq s\leq t} \Vert e^{\lambda s}w(f^\ell-f^{\ell-1})\Vert_{L^\infty_{x,v}} .
\end{align*}
We choose $\delta\ll 1$ such that $C_1\delta < 1$. Then $f^{\ell}$ is a Cauchy sequence, and we construct a solution $f$ to \eqref{f_eqn} such that for all $t>0$, 
\begin{align}
    \Vert e^{\lambda t}wf(t)\Vert_{L^\infty_{x,v}} \leq 2C\delta. \label{f_est}
\end{align}

To prove the uniqueness, we let $f$ and $g$ be two solutions to \eqref{f_eqn} such that $\Vert e^{\lambda t}wf(t)\Vert_{L^\infty_{x,v}}, \Vert e^{\lambda t}wg(t)\Vert_{L^\infty_{x,v}} \leq 2C\delta$. The equation of $f-g$ satisfies
\begin{align*}
\begin{cases}
    &   \p_t (f-g) + v\cdot \nabla_x (f-g) + \mathcal{L}(f-g) = \Gamma(f) - \Gamma(g), \\
    & f(0,x,v)-g(0,x,v) = 0, \\
    & [f-g]|_{\gamma_-} = c_\mu \sqrt{\mu(v)} \int_{n(x)\cdot u>0} [f-g]\sqrt{\mu(u)}(n(x)\cdot u) \dd u.
\end{cases}
\end{align*}
Applying Proposition \ref{prop:linfty}, we have
\begin{align*}
    &   \sup_{0\leq s\leq t}\Vert e^{\lambda s}w(f-g)\Vert_{L^\infty_{x,v}} \leq C_1\delta \sup_{0\leq s\leq t}\Vert e^{\lambda s}w(f-g)\Vert_{L^\infty_{x,v}}.
\end{align*}
Since $C_1\delta<1$, we conclude that $\sup_{0\leq s\leq t}\Vert e^{\lambda s}w(f-g)\Vert_{L^\infty_{x,v}} = 0$, thus $f=g$. We complete the well-posedness.

\textit{Positivity.} Finally, we prove that the unique solution $f$ satisfies $F=\mu+\sqrt{\mu}f\geq 0$. We use a different sequence
\begin{align*}
\begin{cases}
    & \p_t F^{\ell+1} + v\cdot \nabla_x F^{\ell+1} = \nu^{\ell}(M(F^{\ell})-F^{\ell+1}), \\
    & F^{\ell+1}|_{\gamma_-} = c_\mu\mu(v) \int_{n(x)\cdot u>0} F^{\ell}(n(x)\cdot u) \dd u, \\
    & F^{\ell+1}(0,x,v) = F_0(x,v), \ F^{0} = F_0(x,v).    
\end{cases}
\end{align*}
Clearly, such an iteration preserves positivity. In the perturbation $F^{\ell} = \mu + \sqrt{\mu}f^{\ell}$, the equation of $f^{\ell+1}$ reads
\begin{align*}
\begin{cases}
    &  \p_t f^{\ell+1} + v\cdot \nabla_x f^{\ell+1} + \nu^{\ell} f^{\ell+1} =  \mathbf{P}f^\ell + \Gamma_1(f^\ell) + \Gamma_3(f^\ell) + \Gamma_4(f^\ell), \\
    & f^{\ell+1}|_{\gamma_-} = c_\mu\sqrt{ \mu(v)} \int_{n(x)\cdot u>0} f^\ell (n(x)\cdot u)\sqrt{\mu(u)} \dd u, \\
    & f^{\ell+1}(0,x,v) = f_0(x,v), \ f^{0} = f_0(x,v).
\end{cases}
\end{align*}

We prove the following claim: there exists $T^* \ll 1$ such that if the initial condition satisfies $\Vert wf_0\Vert_{L^\infty_{x,v}}<2C\delta$, and $\sup_{i\leq \ell}\sup_{t\leq T^*}\Vert wf^i(t)\Vert_{L^\infty_{x,v}} < 4C\delta \ll 1$, then it holds
\begin{align*}
    &  \sup_{t\leq T^*}\Vert wf^{\ell+1}(t)\Vert_{L^\infty_{x,v}} < 4C\delta \ll 1.
\end{align*}
Here the constant $C$ is constructed in \eqref{f_est}.

\begin{proof}[\text{Proof of claim}]
When $\Vert wf^\ell\Vert_{L^\infty_{x,v}} \ll 1$, from \eqref{macro_control}, the damping factor satisfies $\nu^{\ell}>\frac{1}{2}$. With the estimate of the nonlinear operator $\Gamma_1,\Gamma_3,\Gamma_4$ in Lemma \ref{lemma:gamma_property}, one can employ a similar argument (proof of Proposition \ref{prop:linfty}) and obtain
\begin{align*}
    &   \sup_{t\leq T^*}\Vert wf^{\ell+1}(t)\Vert_{L^\infty_{x,v}} \leq \Vert wf_0\Vert_{L^\infty_{x,v}} + T^* C_1T_0^{5/4}\sup_{i\leq \ell}\sup_{t\leq T^*}\{\Vert wf^i\Vert_{L^\infty_{x,v}}+ \Vert wf^{i}\Vert_{L^\infty_{x,v}}^2 + \Vert wf^i\Vert_{L^\infty_{x,v}}^3\}.
\end{align*}
Here we emphasize that we do not derive the $L^2_{x,v}$ for $\mathbf{P}f^{\ell}$, and directly control such term in the $L^\infty_{x,v}$ estimate using the small time integration $\int_0^{T^*}$. The term $C_1 T_0^{5/4}$ corresponds to the repeat interaction with the boundary in the application of Lemma \ref{lemma:tk}.

By choosing $T^*$ small enough such that 
\begin{align*}
    & \sup_{t\leq T^*}\Vert wf^{\ell +1}\Vert_{L^\infty_{x,v}} \leq \Vert wf_0\Vert_{L^\infty_{x,v}} + \frac{1}{10}\sup_{i\leq \ell}\sup_{t\leq T^*}\{\Vert wf^i\Vert_{L^\infty_{x,v}}+ \Vert wf^{i}\Vert_{L^\infty_{x,v}}^2 + \Vert wf^i\Vert_{L^\infty_{x,v}}^3\} < 4C\delta,
\end{align*}
we conclude the claim.

\end{proof}

Since the $f^0 = f_0$ satisfies the assumption $\Vert wf^0\Vert_{L^\infty_{x,v}}<\delta <2C\delta$, the above claim implies the uniform in $\ell$ estimate: $\sup_{\ell<\infty}\sup_{t\leq T^*}\Vert wf^{\ell}(t)\Vert_{L^\infty_{x,v}} < 2C\delta \ll 1$. The subtraction $f^{\ell+1} - f^{\ell}$ satisfies the equation
\begin{align*}
    \begin{cases}
        &   \p_t (f^{\ell+1}-f^\ell) + v\cdot \nabla_x (f^{\ell+1} - f^\ell ) + \nu^\ell (f^{\ell+1}-f^\ell) = (\nu^{\ell-1}-\nu^{\ell})f^{\ell} + \mathbf{P}(f^\ell - f^{\ell-1}) \\
        & + \Gamma_1(f^\ell) - \Gamma_1(f^{\ell -1})+ \Gamma_3(f^\ell) - \Gamma_3(f^{\ell -1})+ \Gamma_4(f^\ell) - \Gamma_4(f^{\ell -1}), \\
    & [f^{\ell+1}-f^{\ell}]|_{\gamma_-} = c_\mu\sqrt{ \mu(v)} \int_{n(x)\cdot u>0} [f^\ell-f^{\ell-1}] (n(x)\cdot u)\sqrt{\mu(u)} \dd u, \\
    & [f^{\ell+1}-f^{\ell}](0,x,v) = 0.
    \end{cases}
\end{align*}
With the estimate to the difference of the nonlinear operator in Lemma \ref{lemma:nonlinear_substraction}, we apply a similar argument as the claim above and conclude
\begin{align*}
     \sup_{t\leq T^*}\Vert w(f^{\ell+1}-f^\ell)\Vert_{L^\infty_{x,v}}& \leq  T^* C_1T_0^{5/4}(1+C\delta)\max_{i\leq \ell}\sup_{t\leq T^*}\Vert w(f^i-f^{i-1})\Vert_{L^\infty_{x,v}} \\
    &\leq \frac{1}{5}\max_{i\leq \ell}\sup_{t\leq T^*}\Vert w(f^i-f^{i-1})\Vert_{L^\infty_{x,v}}.
\end{align*}
Therefore, $f^{\ell}$ forms a Cauchy sequence in the $L^\infty_{x,v}$ space. By the uniqueness, we conclude the positivity on $[0,T^*]$. Since the unique solution is proved to satisfy $\Vert wf(t)\Vert_{L^\infty_{x,v}}\leq 2Ce^{-\lambda t}\delta$ in \eqref{f_est}, for $[T^*,2T^*],[2T^*,3T^*]...$, we apply the same induction argument as in the proof of claim. Here the only difference is that the initial condition becomes $f^{\ell+1}(nT^*,x,v) = f(nT^*,x,v)$. Since the assumption on the initial condition is still satisfied, we conclude the positivity for any $[nT^*,(n+1)T^*]$. We complete the proof.

\appendix

\section{Proof of Lemma \ref{lemma:L2}}

We derive the $L^2$ dissipation estimate of the macroscopic quantities $a(t,x),\mathbf{b}(t,x),c(t,x)$ using special test functions with the following weak formulation to~\eqref{linear_f_lambda}, here we emphasize that these variables only depend on $t$ and $x$. This method was proposed by \cite{EGKM,EGKM2} for the Boltzmann equation.
\begin{equation}\label{weak_formula_2}
    \begin{split}
        & - \int_0^t \iint_{\O\times \mathbb{R}^3} v\cdot \nabla_x \psi f \dd x \dd v \dd s   \\
    & = \iint_{\O\times \mathbb{R}^3}\{-\psi f(t) + \psi f(0)\} \dd x \dd v +  \int_0^t \iint_{\O\times \mathbb{R}^3} f \p_t \psi \dd x \dd v  \dd s  - \int_0^t \int_{\gamma} \psi f  \dd \gamma  \dd s   \\
    & -  \int_0^t  \iint_{\O\times \mathbb{R}^3} \mathcal{L} f \psi  \dd x \dd v \dd s + \int_0^t \iint_{\O\times \mathbb{R}^3} g \psi \dd  x \dd v \dd s  \\
    & := \{G_\psi(t) - G_\psi(s)\} + J_1 + J_2 + J_3 + J_4 .   
    \end{split}
\end{equation}

\textbf{Step 1: estimate of $c(t,x)$.}

We choose a test function as $\psi_c$ be a solution to the following problem
\begin{align}
\psi  & :=\psi_c = v\cdot \nabla_x \phi_c (|v|^2-5)\mu^{1/2}, \notag \\
& -\Delta \phi_c    = c \text{ in } \O  , \  \phi_c  = 0 \text{ on }\p\O. \label{poisson_c}
\end{align}

From a direct computation, the contribution of $\mathbf{b}$ vanishes from the oddness, and the contribution of $a$ vanishes from the orthogonality of $v(|v|^2-5)\mu^{1/2}\perp \ker \mathcal{L}$. The LHS of~\eqref{weak_formula_2} becomes
\begin{align}
  LHS  &  = 5 \int_0^t \int_{\O} c^2 \dd x\dd s - \sum_{i,j=1}^3 \int_0^t  \int_{\O} \p_{ij}^2 \phi_c  \langle (\mathbf{I}-\mathbf{P})f, v_iv_j (|v|^2-5)\mu^{1/2}\rangle \dd x \dd s \notag\\
  & = 5 \int_0^t\int_\O c^2 \dd x \dd s  + E_1, \label{LHS_c}
\end{align}
where, for any $\delta_1>0$, from the elliptic estimate to \eqref{poisson_c},
\[|E_1|\lesssim \delta_1\int_0^t \Vert c\Vert_{L^2_x}^2\dd s + \frac{1}{ \delta_1}\int_0^t \Vert \mu^{1/4}(\mathbf{I}-\P)f\Vert_{L^2_{x,v}}^2 \dd s.\]
For $J_1$, we denote $\Phi_c$ as the elliptic equation
\begin{align*}
    & -\Delta \Phi_c = \p_t c \text{ in }\O, \ \Phi_c = 0 \text{ on } \p\O.
\end{align*}
Integration by part leads to
\begin{align}
    &    \int_0^t \int_\O |\nabla_x \Phi_c|^2 \dd x \dd s = \int_0^t \int_{\O} \p_t c(s,x)\Phi_c \dd x \dd s. \label{poisson_p_t_c}
\end{align}
Denote $\Lambda_{j}(f):= \frac{1}{10}((|v|^2-5)v_j\sqrt{\mu},f)_v$. From the conservation of energy, we have
\begin{align}
    & \p_t c + \frac{1}{3}\nabla_x \cdot \mathbf{b} + \frac{1}{6} \nabla \cdot  \Lambda((\mathbf{I}-\mathbf{P})f)  = 0. \notag
\end{align}
Then \eqref{poisson_p_t_c} becomes
\begin{align}
    &    \int_0^t \int_\O \p_t c(s,x) \Phi_c \dd x \dd s  = \int_0^t \int_\O \Big[ -\frac{1}{3}\mathbf{b}\cdot \nabla_x \Phi_c - \frac{1}{6} \Lambda((\mathbf{I}-\mathbf{P})f)\cdot \nabla_x \Phi_c \Big] \dd x \dd s \notag \\
    &- \int_0^t \int_{\p\O}  \Big(\frac{1}{3}(\mathbf{b}\cdot n) \Phi_c + \frac{1}{6} (\Lambda((\mathbf{I}-\mathbf{P})f)\cdot n )\Phi_c\Big) \dd S_x \dd s. \label{energy_p_t_c_2}
\end{align}
The boundary term vanishes from the boundary condition $\Phi_c(x) = 0$ on $x\in \p\O$.
The other term in \eqref{energy_p_t_c_2} is controlled as
\begin{align*}
& o(1)\int_0^t \Vert \nabla_x \Phi_c \Vert_{L^2_x}^2 \dd s + \int_0^t \Vert \mathbf{b}\Vert_{L^2_x}^2 \dd s + \int_0^t \Vert (\mathbf{I}-\mathbf{P})f\Vert^2_{L^2_{x,v}} \dd s.
\end{align*}
Plugging the estimates to \eqref{poisson_p_t_c}, we obtain
\begin{align}
    &   \int_0^t \Vert \nabla_x \Phi_c\Vert^2_{L^2_x}\dd s  \lesssim   \int_0^t \Vert \mathbf{b}\Vert_{L^2_x}^2 \dd s + \int_0^t \Vert (\mathbf{I}-\mathbf{P})f\Vert^2_{L^2_{x,v}} \dd s   .   \notag
\end{align}

Thus we compute $J_1$ as
\begin{align}
  |J_1|  &  \lesssim  \delta_1\int_0^t \Vert \nabla_x\Phi_c \Vert^2_{L^2_x} \dd s  + \frac{1}{\delta_1} \int_0^t \Vert (\mathbf{I}-\P)f\Vert^2_{L^2_{x,v}}  \dd s  \lesssim  \delta_1 \int_0^t \Vert \mathbf{b}\Vert_{L^2_x}^2\dd s + \frac{1}{\delta_1} \int_0^t \Vert (\mathbf{I}-\P)f\Vert^2_{L^2_{x,v}}  . \label{J1_c_bdd}
\end{align}

Next, we apply boundary condition of $\phi_c$ and $f$ to compute $J_2$:
\begin{align*}
  \int_\gamma \psi f \dd \gamma  & = \int_{\gamma_+} \psi f \dd \gamma  +  \int_{\gamma_-} \psi f \dd \gamma.
\end{align*}
We have
\begin{align}
 & \int_{\p\O} \bigg[ \int_{n(x)\cdot v>0} + \int_{n(x)\cdot v<0}\bigg] \big(|v|^2 -5\big)\sqrt{\mu} (v\cdot \nabla_x \phi_c) (n\cdot v) f  \dd v \dd S_x \notag \\
 & = \int_{\p\O} \int_{n(x)\cdot v>0} \big(|v|^2 -5 \big) \sqrt{\mu} (v\cdot \nabla_x \phi_c) (n\cdot v)(f-P_\gamma f) \dd v \dd S_x \notag\\
  & + 2 \int_{\p\O}\int_{n(x)\cdot v>0} \big(|v|^2 -5 \big) \sqrt{\mu} |n(x)\cdot v|^2 (n(x)\cdot \nabla_x \Phi_c) P_\gamma f \dd v \dd S_x \notag    \\
  & = \int_{\p\O} \int_{n(x)\cdot v>0} \big(|v|^2 -5 \big) \sqrt{\mu} (v\cdot \nabla_x \phi_c) (n\cdot v)(f-P_\gamma f) \dd v \dd S_x \notag \\
  & \lesssim \delta_1 \Vert \nabla_x \Phi_c\Vert_{L^2(\p\O)}^2 +  \frac{1}{\delta_1} |(I-P_\gamma)f|_{2,+}^2  \lesssim \delta_1 \Vert c\Vert_{L^2_x}^2 + \frac{1}{\delta_1}|(I-P_\gamma)f|_{2,+}^2   .\notag
\end{align}
In the first equality, we have applied the change of variable $v\to v-2(n\cdot v)n$. The third line vanishes by $\int_{n(x)\cdot v>0} (|v|^2-5)(n\cdot v)^2\mu \dd v= 0$. In the last inequality, we applied elliptic estimate to~\eqref{poisson_c} with the trace theorem:
\begin{align*}
    & \Vert \nabla_x \Phi_c \Vert^2_{L^2(\p\O)} \lesssim \Vert \phi_c \Vert^2_{H^2_x} \lesssim \Vert c\Vert_{L^2_x}^2.
\end{align*}

We conclude the estimate for $J_2$ as
\begin{equation}\label{J2_c_bdd}
|J_2| \lesssim \delta_1 \int_0^t \Vert c\Vert_{L^2_x}^2 \dd s + \frac{1}{\delta_1}\int_0^t |(I-P_\gamma)f|_{2,+}^2  \dd s.
\end{equation}

For $J_3$, due to the exponential decay factor $\mu^{1/2}$ in $\phi_c$, we have
\begin{align}
  |J_3|  &  \lesssim  \delta_1 \int_0^t \Vert c\Vert_{L^2_x}^2 \dd s + \frac{1}{\delta_1}\int_0^t \Vert \mu^{1/4} \mathcal{L} (\mathbf{I}-\mathbf{P})f\Vert_{L^2_{x,v}}^2 \dd s \notag\\
  &\lesssim  \delta_1 \int_0^t \Vert c\Vert_{L^2_x}^2 \dd s + \frac{1}{\delta_1} \int_0^t \Vert (\mathbf{I}-\mathbf{P})f\Vert_{L^2_{x,v}}^2 \dd s . \label{J3_c_bdd}
\end{align}
Here we applied the elliptic estimate to $\nabla_x \phi_c$.

For $J_4$, similar to the computation in~\eqref{J3_c_bdd}, we have
\begin{align}
  \Big|\int_0^t \iint_{\O\times \mathbb{R}^3} g\psi_c\dd x \dd v \dd s\Big|  & \lesssim \delta_1 \int_0^t \Vert c\Vert_{L^2_x}^2 \dd s + \frac{1}{\delta_1} \int_0^t \Vert g\Vert_{L^2_{x,v}}^2 \dd s. \label{J4_c_bdd}
\end{align}

Collecting~\eqref{LHS_c}, \eqref{J1_c_bdd}, \eqref{J2_c_bdd}, \eqref{J3_c_bdd} and \eqref{J4_c_bdd}, we conclude the estimate of $c$ as follows: for some $C_1>0$ and $G_c(t):= \int_\O \int_{\mathbb{R}^3} \psi_c f(t)\dd x \dd v$,
\begin{align}
  \int_0^t \Vert c\Vert_{L^2_x}^2 \dd s & \leq C_1\Big[ G_c(t) - G_c(0) + \delta_1 \int_0^t \Vert \mathbf{b}\Vert_{L^2_x}^2 \dd s + \frac{1}{\delta_1} \int_0^t \Vert (\mathbf{I}-\mathbf{P})f\Vert_{L^2_{x,v}}^2 \dd s \notag\\
  & + \frac{1}{\delta_1}\int_0^t |(I-P_\gamma)f|_{2,+}^2 \dd s +  \frac{1}{\delta_1} \int_0^t \Vert g\Vert_{L^2_{x,v}}^2 \dd s\Big]. \label{c_bdd}
\end{align}

\textbf{Step 2: estimate of $\mathbf{b}(t,x)$.}
We use the weak formulation in~\eqref{weak_formula_2} for the estimate of $\mathbf{b}$. 

First, we estimate $b_1$. We choose a test function as
\begin{align}
    & \psi_1 =  \frac{3}{2}\Big(|v_1|^2 - \frac{|v|^2}{3} \Big)\sqrt{\mu} \p_{x_1} \phi_1  + v_1v_2 \sqrt{\mu} \p_{x_2} \phi_1 + v_1 v_3 \sqrt{\mu} \p_{x_3} \phi_1. \notag
\end{align}

We let $\phi_1$ satisfy the elliptic system
\begin{align}
\begin{cases}
        & -\p_{x_1}^2 \phi_1  -\Delta\phi_1 = b_1 \text{ in }\O, \\
    &  \phi_1 = 0 \text{ on } x\in \p\O. 
\end{cases}\label{phi_b_k}
\end{align}
From a direct computation, the contribution of $a$ and $c$ vanish from the oddness, and the LHS of~\eqref{weak_formula_2} becomes
\begin{align*}
    &    -\int_0^t \int_{\mathbb{R}^3}\int_{\O} \Big[\frac{3}{2}v_1^2 \Big(|v_1|^2 - \frac{|v|^2}{3} \Big) \mu \p_{11} \phi_1 b_1 + v_1^2v_2^2 \mu \p_{12} \phi_1 b_2  + v_1^2 v_3^2  \mu\p_{13} \phi_1 b_3 \\
    & + v_1^2 v_3^2 \mu \p_{33} \phi_1 b_1 + v_1^2 v_2^2 \mu \p_{22}\phi_1 b_1 + \frac{3}{2}v_2^2\Big(|v_1|^2-\frac{|v|^2}{3}\Big) \mu \p_{12}\phi_1 b_2  + \frac{3}{2} v_3^2 \Big( |v_1|^2 - \frac{|v|^2}{3}\Big) \mu \p_{13} \phi_1 b_3\Big] \dd x \dd v\dd s + E_2\\
    & = -\int_0^t \int_{\mathbb{R}^3}\int_\O [2 \p_{11}\phi_1 b_1+ \p_{22}\phi_1 b_1 +  \p_{33}\phi_1 b_1 + \p_{12}\phi_1b_2 - \p_{12}\phi_1 b_2 +  \p_{13}\phi_1 b_3  - \p_{13}\phi_1 b_3 ] \dd x \dd v \dd s + E_2  \\
    &= - \int_0^t (\Delta \phi_1 + \p_{11}\phi_1) b_1 \dd x \dd s +E_2= \int_0^t \Vert b_1\Vert_{L^2_x}^2 \dd s + E_2.
 \end{align*}

For $E_{2}$, using the elliptic estimate, we obtain that for any $\delta_2\ll 1$,
\begin{align*}
    &  |E_{2}| \lesssim \delta_2 \int_0^t \Vert b_1\Vert_{L^2_x}^2 \dd s + \frac{1}{\delta_2} \int_0^t \Vert \mu^{1/4} (\mathbf{I}-\mathbf{P})f\Vert^2_{L^2_{x,v}}  \dd s   .
\end{align*}

For $J_1$, we let $\Phi_1$ satisfy the elliptic equation
\begin{align*}
    &  -\p_{x_1}^2 \Phi_1 - \Delta \Phi_1 = \p_t b_1 \text{ in }\O, \ \Phi_1 = 0 \text{ on }x\in \p\O.
\end{align*}
Integration by part leads to
\begin{align}
    &   \int_0^t \int_\O [2|\p_{x_1}\Phi_1|^2 + |\p_{x_2}\Phi_1|^2 + |\p_{x_3} \Phi_1|^2 ] \dd x \dd s    = \int_0^t \int_\O \p_t b_1 \Phi_1 \dd x \dd s.   \label{energy_p_t_b}
\end{align}
Denote $\Theta_{ij}(f):= ((v_iv_j-1)\sqrt{\mu},f)_v$. From the conservation of momentum, we have
\begin{align}
    & \p_t b_1 + \p_{x_1}(a+2c) + \nabla_x \cdot \Theta_{1}((\mathbf{I}-\mathbf{P})f)  = 0.   \notag
\end{align}
Then \eqref{energy_p_t_b} becomes
\begin{align}
    &    \int_0^t \int_\O \p_t b_1 \Phi_1 \dd x \dd s  = \int_0^t \int_\O  \Big[ (a+2c)\p_{x_1}\Phi_1 + \Theta_1((\mathbf{I}-\mathbf{P})f)\cdot \nabla_x \Phi_1 \Big] \dd x \dd s \notag \\
    & - \int_0^t  \int_{\p\O}  \Phi_1 (a+2c)n_1 + \Phi_1 (\Theta_{1}((\mathbf{I}-\mathbf{P})f)\cdot n) \dd S_x \dd s. \label{energy_p_t_b_2}
\end{align}
The boundary term vanishes from the boundary condition $\Phi_1(x) = 0,x\in \p\O$.

The other term in \eqref{energy_p_t_b_2} is controlled as
\begin{align*}
    & o(1)\int_0^t \Vert \nabla_x \Phi_1\Vert_{L^2_x}^2 \dd s + \int_0^t [\Vert a\Vert_{L^2_x}^2 + \Vert c\Vert_{L^2_x}^2 + \Vert (\mathbf{I}-\mathbf{P})f\Vert_{L^2_{x,v}}^2] \dd s.
\end{align*}
Plugging this estimate to \eqref{energy_p_t_b}, we obtain
\begin{align}
    &  \int_0^t \Vert \nabla_x \Phi_1\Vert_{L^2_x}^2 \dd s  \lesssim \int_0^t [\Vert a\Vert_{L^2_x}^2 + \Vert c\Vert_{L^2_x}^2 + \Vert (\mathbf{I}-\mathbf{P})f\Vert_{L^2_{x,v}}^2] \dd s.  \label{Phi_b_estimate}
\end{align}

$J_1$ can be computed using the estimate \eqref{Phi_b_estimate}:
\begin{align*}
    &   |J_1| \lesssim \delta_2 \int_0^t \Vert \nabla_x \Phi_1 \Vert_{L^2_x}^2 \dd s + \int_0^t \Vert (\mathbf{I}-\mathbf{P})f\Vert_{L^2_{x,v}}^2 \dd s \\ 
    & \lesssim   \delta_2 \int_0^t [\Vert a\Vert^2_{L^2_x} + \Vert c\Vert_{L^2_x}^2]  + \int_0^t \Vert c\Vert^2_{L^2_x} \dd s + \int_0^t \Vert (\mathbf{I}-\mathbf{P})f\Vert_{L^2_{x,v}}^2 \dd s   . 
\end{align*}

Next we compute the boundary integral $J_2$ using the diffuse boundary condition:
\begin{align*}
    & \int_{\mathbb{R}^3} \int_{\p\O} f\psi_1 (n(x)\cdot v) \dd S_x \dd v =  \int_{\mathbb{R}^3}\int_{\p\O} P_\gamma f \psi_1 (n(x)\cdot v) \dd S_x \dd v + \int_{\p\O}\int_{n(x)\cdot v>0} (I-P_\gamma)f \psi_1 (n(x)\cdot v) \dd v \dd S_x \\
    &  \lesssim |(I-P_\gamma)f|_{L^2_{\gamma_+}}^2 + o(1) |\nabla_x \Phi_1|_{L^2(\p\O)}^2  \lesssim |(I-P_\gamma)f|_{L^2_{\gamma_+}}^2 + o(1)\Vert \phi\Vert^2_{H^2_x} \lesssim |(I-P_\gamma)f|_{L^2_{\gamma_+}}^2 + o(1)\Vert b_1\Vert_{L^2_{x}}^2 .
\end{align*}
In the second line, the contribution of $P_\gamma f$ vanished due to oddness. In the third line, we applied the trace theorem with the Poincare inequality.

$J_3$ and $J_4$ are bounded in a similar manner as \eqref{J3_c_bdd} and \eqref{J4_c_bdd}:
\begin{align*}
    &  |J_3| \lesssim \delta_2  \int_0^t \Vert b_1\Vert^2_{L^2_x} + \frac{1}{\delta_2} \int_0^t \Vert (\mathbf{I} - \mathbf{P})f\Vert^2_{L^2_{x,v}} \dd s, \\
    & |J_4| \lesssim \delta_2 \int_0^t \Vert b_1 \Vert^2_{L^2_x} \dd s  + \frac{1}{\delta_2} \int_0^t \Vert g\Vert^2_{L^2_{x,v}} \dd s.
\end{align*}

For $G_{b_1}(t):= \int_\O \int_{\mathbb{R}^3} \psi_1 f(t)\dd x \dd v$, we conclude the following for $b_1$:
\begin{align}
     \int_0^t \Vert b_1\Vert_{L^2_x}^2 \dd s&      \lesssim G_{b_1}(t)-G_{b_1}(0) + \delta_2 \int_0^t \Vert (a,b_1,c)\Vert^2_{L^2_x}  \notag \\
    &+ \frac{1}{\delta_2} \int_0^t \Vert (\mathbf{I}-\mathbf{P})f\Vert_{L^2_{x,v}}^2 \dd s + \frac{1}{\delta_2} \int_0^t |(I-P_\gamma)f|_{2,+}^2 \dd s .\notag  
\end{align}
The estimate to $b_2$ and $b_3$ are the same by modifying the test function \eqref{phi_b_k} to the following:
\begin{align*}
    &\psi_2 =  v_1v_2\sqrt{\mu}\p_{x_1}\phi_2 + \frac{3}{2}\Big(|v_2|^2 - \frac{|v|^2}{3} \Big)\sqrt{\mu} \p_{x_2} \phi_2  + v_2 v_3 \sqrt{\mu} \p_{x_3} \phi_2, \\
    &  -\p_{x_2}^2 \phi_2  -\Delta\phi_2 = b_2, \\
    &\psi_3 =  v_1v_3\sqrt{\mu}\p_{x_1}\phi_3 +  v_2 v_3 \sqrt{\mu} \p_{x_2} \phi_3 + \frac{3}{2}\Big(|v_3|^2 - \frac{|v|^2}{3} \Big)\sqrt{\mu} \p_{x_3} \phi_3   ,\\
    &  -\p_{x_3}^2 \phi_3  -\Delta\phi_3 = b_3.
\end{align*}
For $G_{b}(t):= \int_\O \int_{\mathbb{R}^3} (\psi_1+\psi_2+\psi_3) f(t)\dd x \dd v$, we conclude the estimate for $\mathbf{b}$ as follows: for some $C_2$ and any $\delta_2>0$,
\begin{align}
 \int_0^t \Vert \mathbf{b}\Vert_{L^2_x}^2 \dd s   &      \lesssim C_2\Big[  G_b(t)-G_b(0) + \delta_2 \int_0^t \Vert (a,c)\Vert^2_{L^2_x} \notag\\
   & + \frac{1}{\delta_2}\int_0^t \Vert g\Vert_{L^2_{x,v}} \dd s + \frac{1}{\delta_2} \int_0^t \Vert (\mathbf{I}-\mathbf{P})f\Vert_{L^2_{x,v}}^2 \dd s + \frac{1}{\delta_2}\int_0^t |(I-P_\gamma)f|_{2,+}^2 \dd s\Big] .  \label{b_estimate}
\end{align}

\textbf{Step 3: estimate of $a(t,x)$.}

We choose the test function as
\begin{align}
&\psi=\psi_a := \sum_{i=1}^3  \p_i \phi_a  v_i (|v|^2 - 10) \mu^{1/2} ,\notag\\
& -\Delta \phi_a = a, \ \text{ in }\O, \ \nabla_x \phi_a \cdot n = 0 \text{ on } \p\O. \label{phi_a}
\end{align}

From direction computation, on LHS of~\eqref{weak_formula_2}, $\mathbf{b}$ vanished from oddness and $c$ vanished from $\int_{\mathbb{R}^3}(|v|^2-10)v_i^2(|v|^2-3)\dd v =0$. Thus we obtain 
\begin{align}
  LHS  &  = 5 \int_0^t \Vert a\Vert_{L^2_x}^2\dd s - \sum_{i,j=1}^3 \int_0^t \int_{\O} \p_{ij}^2 \phi_a \langle v_iv_j (|v|^2-10)\mu^{1/2},(\mathbf{I}-\mathbf{P})f\rangle \dd x \dd s \notag\\
  & = 5 \int_0^t \Vert a\Vert_{L^2_x}^2\dd s + E_3, \label{LHS_a}
\end{align}
where, for any $\delta_3>0$,
\begin{align*}
  |E_3|  & \lesssim \delta_3 \int_0^t \Vert a\Vert_{L^2_x}^2\dd s  +  \frac{1}{\delta_3} \int_0^t \Vert \mu^{1/4}(\mathbf{I}-\P)f\Vert^2_{L^2_{x,v}} \dd s.
\end{align*}
For $J_1$ in~\eqref{weak_formula_2}, we denote
\begin{align}
    & -\Delta \Phi_a = \p_t a(s), \ \text{in } \O, \ \nabla_x \Phi_a \cdot n = 0 \text{ on } \p\O.  \notag
\end{align}
Integration by part leads to
\begin{align}
    &   \int_0^t \int_{\O} |\nabla_x \Phi_a|^2 \dd x \dd s   = \int_0^t \int_{\p\O} \p_t a \Phi_a  \dd x \dd s.   \label{energy_p_t_a}
\end{align}
From the conservation of mass $\p_t a + \nabla_x \cdot \mathbf{b}= 0,$ it holds
\begin{align}
    &    \int_0^t \int_{\O} \p_t a \Phi_a \dd x \dd s  = \int_0^t \int_{\O} \mathbf{b}\cdot \nabla_x \Phi_a \dd x \dd s  - \int_0^t \int_{\p\O} (\mathbf{b}\cdot n) \Phi_a \dd S_x \dd s. \label{energy_p_t_a_2}
\end{align}
The boundary term can be computed as
\begin{align*}
    & \int_0^t \int_{\p\O} (\mathbf{b}\cdot n) \Phi_a \dd S_x \dd s\\
    & =  \int_0^t  \Phi_a \Big[ \int_{\p\O}\int_{n(x)\cdot v>0} (n(x)\cdot v) \sqrt{\mu} (P_\gamma f + (I-P_\gamma)f)  \dd v \dd S_x + \int_{\p\O}\int_{n(x)\cdot v<0} (n(x)\cdot v) \sqrt{\mu} P_\gamma f \dd v \dd S_x\Big] \dd s\\
    & \lesssim o(1)\int_0^t |\Phi_a|_{L^2(\p\O)}^2 \dd s + \int_0^t |(I-P_\gamma)f|^2_{2,+} \dd s\lesssim o(1)\int_0^t \Vert \nabla_x \Phi_a\Vert_{L^2_x}^2 \dd s + \int_0^t |(I-P_\gamma)f|_{2,+}^2 \dd s \\
    & \lesssim o(1)\int_0^t \Vert \nabla_x \phi_a\Vert_{L^2_x}^2 \dd s + \int_0^t |(I-P_\gamma)f|_{2,+}^2 \dd s .
\end{align*}
In the third line, the contribution of $P_\gamma f$ vanished from the oddness, and we applied the trace theorem. In the last line, we applied the Poincaré inequality.

The other term in \eqref{energy_p_t_a_2} is controlled as
\begin{align*}
o(1) \int_0^t \Vert \nabla_x \Phi_a \Vert_{L^2_x}^2  \dd s +  \int_0^t \Vert \mathbf{b}\Vert_{L^2_x}^2  \dd s.
\end{align*}
Plugging the estimates to \eqref{energy_p_t_a}, we obtain
\begin{align}
    &   \int_0^t \Vert \nabla_x \Phi_a \Vert_{L^2_x}^2 \dd s \lesssim \int_0^t \Vert \mathbf{b}\Vert_{L^2_x}^2  \dd s + \int_0^t |(I-P_\gamma)f|_{2,+}^2 \dd s.   \label{Phi_a_estimate}
\end{align}

We apply \eqref{Phi_a_estimate} to compute $J_1$ as
\begin{align}
  |J_1|  &  \lesssim  \int_0^t \Vert \nabla_x \Phi_a \Vert^2_{L^2_x} \dd s +  \int_0^t \Vert \mathbf{b}\Vert_{L^2_x}^2 \dd s +  \int_0^t \Vert \mu^{1/4}(\mathbf{I}-\P)f\Vert^2_{L^2_{x,v}}  \dd s \notag\\
  & \lesssim  \int_0^t \Vert \mathbf{b}\Vert_{L^2_x}^2\dd s  +  \int_0^t \Vert \mu^{1/4}(\mathbf{I}-\P)f\Vert^2_{L^2_{x,v}}  \dd s + |(I-P_\gamma) f|^2_{2,+}. \label{J1_a_bdd}
\end{align}

Then we apply the boundary condition of $\phi_a$ and $f$ to compute $J_2$:
\begin{align*}
    &  \int_{\gamma} \psi f \dd \gamma = \int_{\gamma_+} \psi f \dd \gamma + \int_{\gamma_-} \psi f \dd \gamma.
\end{align*}
We compute that
\begin{align*}
    & \int_{\p\O} \Big[\int_{n(x)\cdot v>0} + \int_{n(x)\cdot v<0} \Big] (|v|^2-10) \mu^{1/2} (v\cdot \nabla_x \phi_a) (n\cdot v) f \dd v \dd S_x \notag \\
    &  = \int_{\p\O} \int_{n(x)\cdot v>0}   (|v|^2 - 10) \mu^{1/2} (v\cdot \nabla_x \phi_a) (n\cdot v) (f-P_\gamma f) \dd v \dd S_x  \notag\\
    & + 2\int_{\p\O} \int_{n(x)\cdot v>0} (|v|^2 - 10)\mu^{1/2} (n\cdot \nabla_x \phi_a) (n\cdot v)^2 P_\gamma f \dd v\dd S_x  \notag\\
    & = \int_{\p\O} \int_{n(x)\cdot v>0}   (|v|^2 - 10) \mu^{1/2} (v\cdot \nabla_x \phi_a) (n\cdot v) (f-P_\gamma f) \dd v \dd S_x  \notag\\
    & \lesssim \delta_3 \Vert \nabla_x \Phi_a\Vert_{L^2(\p\O)}^2 + \frac{1}{\delta_3} |(I-P_\gamma)f|^2_{2,+}\lesssim \delta_3 \Vert a\Vert_{L^2_x}^2 + \frac{1}{\delta_3} |(I-P_\gamma)f|^2_{2,+}. 
\end{align*}
In the first equality, we used the change of variable $v\to v-2(n(x)\cdot v)n(x)$. In the second equality, the third line vanishes due to the boundary condition of $\phi_a$ in \eqref{phi_a}. In the last inequality, we used the standard elliptic estimate of \eqref{phi_a} with the trace theorem: $\Vert \phi_a\Vert_{H^1(\p\O)}\lesssim \Vert \phi_a\Vert_{H^2_x}\lesssim \Vert a\Vert_{L^2_x}$.

We derive the estimate for $J_2$ as
\begin{equation}\label{J2_a_bdd}
|J_2| \lesssim \delta_3 \int_0^t \Vert a\Vert_{L^2_x}^2 \dd s + \frac{1}{\delta_3} \int_0^t |(I-P_\gamma )f|_{2,+}^2 \dd s.
\end{equation}

$J_3$ and $J_4$ are estimated similarly as \eqref{J3_c_bdd} and \eqref{J4_c_bdd}:
\begin{equation}\label{J3_a_bdd}
|J_3| \lesssim \delta_3 \int_0^t \Vert a\Vert_{L^2_x}^2 \dd s + \frac{1}{\delta_3} \int_0^t \Vert (\mathbf{I}-\mathbf{P})f \Vert^2_{L^2_{x,v}} \dd s,
\end{equation}
\begin{equation}\label{J4_a_bdd}
|J_4| \lesssim \delta_3 \int_0^t \Vert a\Vert_{L^2_x}^2 \dd s + \frac{1}{\delta_3} \int_0^t \Vert g\Vert_{L^2_{x,v}}^2 \dd s.
\end{equation}

Collecting \eqref{LHS_a}, \eqref{J1_a_bdd}, \eqref{J2_a_bdd}, \eqref{J3_a_bdd} and \eqref{J4_a_bdd}, we conclude the estimate $a$ as follows: for some $C_3>0$ and $G_a(t):= \int_\O \int_{\mathbb{R}^3} \psi_a f(t)\dd x \dd v$
\begin{align}
  \int_0^t \Vert a\Vert_{L^2_x}^2 \dd s  & \leq C_3 \Big[ G_a(t)-G_a(0) + \int_0^t \Vert \mathbf{b}\Vert_{L^2_x}^2 \dd s  + \frac{1}{\delta_3} \int_0^t \Vert (\mathbf{I}-\mathbf{P})f\Vert_{L^2_{x,v}}^2 \dd s \notag\\
  & +  \frac{1}{\delta_3} \int_0^t \Vert g\Vert_{L^2_{x,v}}^2 \dd s + \frac{1}{\delta_3}\int_0^t |(I-P_\gamma )f|_{2,+}^2 \dd s\Big]. \label{a_bdd}
\end{align}

\textbf{Step 4: conclusion}  

We summarize \eqref{a_bdd}, \eqref{b_estimate} and \eqref{c_bdd}. We let $\delta_2=\sqrt{\delta_1}$, and multiply \eqref{b_estimate} by $\delta_1^{3/4}$ to have
\begin{align}
 \delta_1^{3/4} \int_0^t \Vert \mathbf{b}\Vert_{L^2_x}^2 \dd s  & \leq  C_2 \delta_1^{5/4} \int_0^t \Vert a\Vert_{L^2_x}^2 \dd s + C_2 \delta_1^{1/4} \int_0^t \Vert c\Vert^2_{L^2_x} \dd s   + C_2 \delta_1^{3/4} \Big[ G_b(t) - G_b(0)  \notag\\
  & +  \frac{1}{\sqrt{\delta_1}} \int_0^t \Vert (\mathbf{I}-\mathbf{P})f\Vert_{L^2_{x,v}}^2 \dd s + \frac{1}{\sqrt{\delta_1}}\int_0^t \Vert g\Vert_{L^2_{x,v}}^2 \dd s + \frac{1}{\sqrt{\delta_1}} \int_0^t |(I-P_\gamma)f|^2_{2,+} \dd s  \Big]. \label{b_estimate_delta}
\end{align}

Then we evaluate $\delta_1 \times \eqref{a_bdd} + \eqref{b_estimate_delta} + \eqref{c_bdd}   $ as
\begin{align*}
    &  \delta_1 \int_0^t \Vert a\Vert^2_{L^2_x} \dd s  +  \delta_1^{3/4} \int_0^t \Vert \mathbf{b}\Vert_{L^2_x}^2 \dd s +  \int_0^t \Vert c\Vert_{L^2_x}^2 \dd s \\
    & \leq   (C_3 \delta_1 + C_1 \delta_1) \int_0^t \Vert \mathbf{b}\Vert_{L^2_x}^2 \dd s + C_2 \delta_1^{5/4} \int_0^t \Vert a\Vert_{L^2_x}^2 \dd s + C_2 \delta_1^{1/4} \int_0^t \Vert c\Vert_{L^2_x}^2 \dd s \\
    & + C \Big[G_a(t)+G_b(t)+G_c(t) - G_a(0)-G_b(0)-G_c(0)  \Big] \\
    & + C \Big[\int_0^t \Vert (\mathbf{I}-\mathbf{P})f\Vert_{L^2_{x,v}}^2 \dd s + \int_0^t \Vert g\Vert_{L^2_{x,v}}^2 \dd s +  \int_0^t |(I-P_\gamma)f|^2_{2,+} \dd s  \Big].
\end{align*}
Here the constant $C$ in the last two lines depends on $C_1,C_2,C_3,\delta_1$. We choose small enough $\delta_1$ such that 
\begin{align*}
    &C_3\delta_1+C_1\delta_1 < \delta_1^{3/4}, \ \ C_2 \delta_1^{5/4} < \delta_1, \ \ C_2 \delta_1^{1/4}<1.
\end{align*}
Finally, we conclude the lemma with $|G(t)| = |G_a(t) + G_b(t) + G_c(t)| = |\iint_{\O\times \mathbb{R}^3} (\psi_a + \psi_b +\psi_c) f(t) \dd x \dd v |  \lesssim \Vert f(t)\Vert_{L^2_{x,v}}^2$.

\hide
\section{Hydrodynamic limit}

To derive the incompressible fluid limit, we focus on the scaled equation
\begin{align*}
    &   \e \p_t F + v\cdot \nabla_x F = \frac{M(F)-F}{\e}, \\
    & F = \mu + \e \sqrt{\mu} f_1 + \e^2 \sqrt{\mu} f_2 + \e^{3/2} \sqrt{\mu} R.
\end{align*}
Then the BGK operator becomes
\begin{align*}
    & M(F) = \mu + \e \sqrt{\mu} (\mathbf{P}f_1 + \e \mathbf{P}f_2 + \e^{1/2}\mathbf{P}R) \\
    &+ \sum_{0\leq i,j\leq 4} \Big(\int_0^1 \{D^2_{(\rho_\vartheta,\rho_\vartheta U_\vartheta, G_\vartheta)}\}M(\vartheta)(1-\vartheta)\dd \vartheta \Big) \langle \e f_1 + \e^2 f_2 + \e^{3/2} R, \chi_i \rangle \langle \e f_1 + \e^2 f_2 + \e^{3/2} R, \chi_j \rangle .
\end{align*}

Collecting the order of $\e$, we obtain
\begin{align}
    &     \text{Order $O(1)$:} \ \mathbf{P}f_1 - f_1 = 0,  \label{order:1}\\
    &     \text{Order $O(\e)$:} \  v\cdot \nabla_x f_1 + f_2 - \mathbf{P}f_2 = \sum_{0\leq i,j\leq 4} \Big(\int_0^1 \{D^2_{(\rho_\vartheta,\rho_\vartheta U_\vartheta, G_\vartheta)}\}M(\vartheta)(1-\vartheta)\dd \vartheta \Big) \langle f_1,\chi_i\rangle \langle f_1,\chi_j\rangle ,  \label{order:e}\\
    & \text{Order $O(\e^2)$:} \ \p_t f_1 + v\cdot \nabla_x f_2  \notag\\
    & = \sum_{0\leq i,j\leq 4} \Big(\int_0^1 \{D^2_{(\rho_\vartheta,\rho_\vartheta U_\vartheta, G_\vartheta)}\}M(\vartheta)(1-\vartheta)\dd \vartheta \Big) (\langle f_1,\chi_i\rangle \langle f_2,\chi_j\rangle + \langle f_2,\chi_i\rangle \langle f_1,\chi_j\rangle ) .  \label{order:e2}
\end{align}
\eqref{order:1} implies that
\begin{align*}
    & f_1 = \Big(\rho_1 + u_1\cdot v + \theta_1 \frac{|v|^2-3}{2} \Big)\sqrt{\mu}.
\end{align*}

Moments of \eqref{order:e} leads to the incompressibility condition and Bousinesq relation
\begin{align*}
    &    \nabla_x \cdot u_1 = 0, \ \nabla_x(\rho_1 + \theta_1) = 0.
\end{align*}

\unhide

\ \\
\noindent{\bf Acknowledgements}
Marlies Pirner was funded by the Deutsche Forschungsgemeinschaft (DFG, German Research Foundation) under Germany’s Excellence Strategy EXC 2044-390685587,
Mathematics Muenster: Dynamics–Geometry–Structure, and the German Science Foundation DFG (grant no. PI 1501/2-1). \\

\noindent{\bf Conflict of Interest:} The authors declare that they have no conflict of interest.

\bibliographystyle{siam}
\bibliography{citation}

\end{document}